\documentclass[a4paper, 10pt]{extarticle}
\usepackage[top=70pt,bottom=70pt,left=75pt,right=77pt]{geometry}

\title{Farey Sequences for Thin Groups}
\author{
{\sc Christopher Lutsko$^*$}
\\[8pt]
$^*$University of Bristol, UK
} 

\usepackage{amssymb, amsmath, amsthm, setspace, bbm, prettyref, graphicx, float, subfigure, color, mathrsfs, mathtools}
\usepackage{Lutsko_Style}

\definecolor{dblue}{rgb}{0.09,0.32,0.44} 
\usepackage[colorlinks=true, linkcolor=dblue, citecolor=dblue, urlcolor=blue]{hyperref}

\newcommand{\BMS}{\mathrm{m}^{BMS}}
\newcommand{\BR}{\mathrm{m}^{BR}}

\begin{document}

  \maketitle
  \begin{abstract}
    \noindent
    The classical Farey sequence of height $Q$ is the set of rational numbers in reduced form with denominator less than $Q$. In this paper we introduce the concept of a generalized Farey sequence. While these sequences arise naturally in the study of discrete (and in particular thin) subgroups, they can be used to study interesting number theoretic sequences - for example rationals whose continued fraction partial quotients are subject to congruence conditions. We show that these sequences equidistribute, that the gap distribution converges, and we answer an associated problem in Diophantine approximation with Fuchsian groups. Moreover, for one specific example, we use a sparse Ford configuration construction to write down an explicit formula for the gap distribution. Finally for this example, we construct the analogue of the Gauss measure in this context which we show is ergodic for the Gauss map. This allows us to prove a theorem about the Gauss-Kuzmin statistics of the sequence.
    \medskip\noindent

        \medskip\noindent
            {\sc Key words and phrases:} 
            Farey Sequence; Continued Fractions; Equidistribution; Local Statistics; Ford Circles; Patterson-Sullivan Theory.

  \end{abstract}

  \section{Introduction} \label{sec:Introduction}

Consider the classical \emph{Farey sequence} of height $Q$:

\begin{equation}
  \tilde{\mathcal{F}}_Q := \left\{ \frac{p}{q} \in [0,1)\;:\; (p,q) \in \hat{\Z}^2, 0<q < Q\right\},
\end{equation}
where $\hat{\Z}^2$ denotes the set of primitive vectors in $\Z^2$. Naturally this sequence is a fundamental object in number theory dating back to 1802 with its introduction by Haros and subsequent work by Farey and Cauchy. For example, this sequence has connections to the Riemann hypothesis (see for example \cite{LagariasMehta2017}) and plays a fundamental role in Diophantine approximation. 

In this paper, we generalize the Farey sequence. For concreteness, one example of such a generalized Farey sequence is given by the following: throughout the paper we use the standard continued fration notation

\begin{equation}
  [a_0;a_1,\dots a_n] = a_0 + \frac{1}{a_1 + \frac{1}{a_2 + \frac{...}{a_n}}}
\end{equation}
(see for example \cite{Khinchin2003}) then denote

\begin{equation}\label{Q4}
   \cQ_4 : = \set{ [0;a_1,\dots a_k] : k \in \N \; , \; a_i \in 4\Z_{\neq 0} \; \forall i },
\end{equation}
that is, rationals whose continued fraction expansions involve only multiples (\emph{possibly negative}) of $4$. The generalized Farey sequence in this context is

\begin{equation}\label{Fhat}
  \wh{\cF}_{Q} = \{\frac{p}{q} \in \cQ_4 : 0 < q < Q, \; \gcd (p,q)=1 \}
\end{equation}
we return to this example in Section \ref{ss:Circle Packing Example} where we give a geometric interpretation of these sets. To see some of the points of $\cQ_4$ see \prettyref{fig:example}.


\begin{figure}[ht!]
  \begin{center}    
    \includegraphics[width=0.9\textwidth]{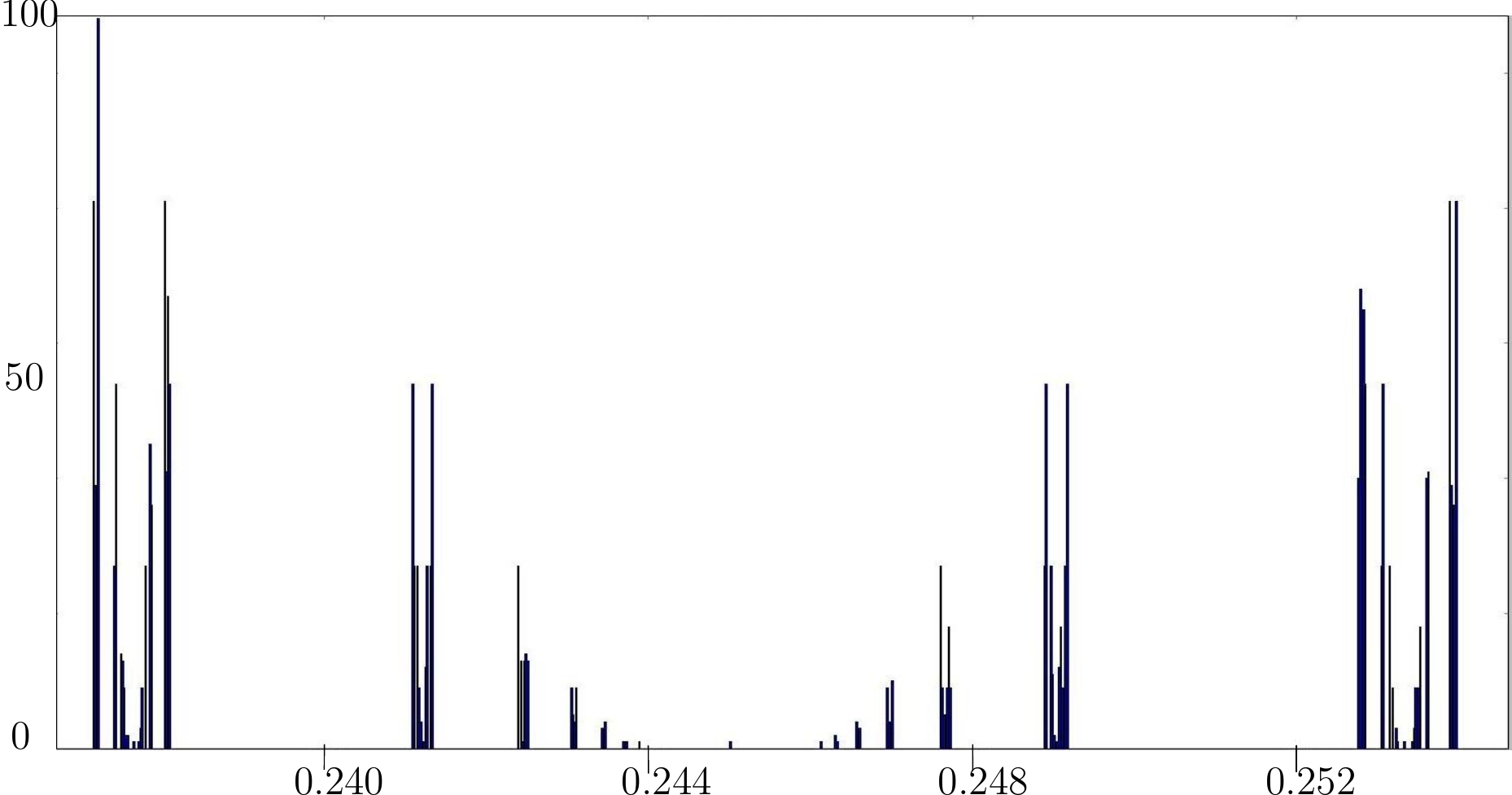}
  \end{center}
  \caption{%
     {\tt Above we show some of the points in $\cQ_4$. The graph was generated as follows: we generated all words of length $10$ (with respect to the two generators applied to $\infty$). Then separated the interval $[0,1)$ into bins of size $10^{-5}$. The above is a bar chart showing the number of points in each bin. Note that the sequence is supported on a fractal subset of the interval. This does not show $\wh{\cF}_{Q}$ (as the cut-off is with respect to word length), however will suffice for a qualitative picture.
       }
   }%
  
  \label{fig:example}

\end{figure}

There is a geometric interpretation of the classical Farey sequence which will play an integral role in this paper. Consider the groups $G=\PSL(2,\R)$ and $\Lambda := \PSL(2,\Z)< G$. $G$ acts on the hyperbolic half-space, $\half$ via M\"{o}bius transformations (see Section \ref{sec:Background Hyp}). As $\Lambda$ is a lattice, there exists a tessalation of $\half$ into disjoint, finite volume subsets such that $\Lambda$ acts transitively on them. These \emph{fundamental domains} are not  compact as each one contains a point on the boundary $\partial \half = \R \cup \{\infty\}$, at the end of a cusp. The set of such cuspidal points is exactly

\begin{equation}
  (\Lambda/\Lambda_{\infty}) \infty = \Q
\end{equation}
(we use $G_x$ to denote the stabilizer of $x$ in a group $G$). That is, the set of cuspidal points can be written as the $\Lambda$-orbit of the point at $\infty\in \partial \half$ - which corresponds to the rationals. Thus the Farey sequence of height $Q$ can be written 

\begin{equation}
  \tilde{\cF}_Q = \set{ \frac{p}{q} \in (\Lambda/\Lambda_{\infty})\infty : (p,q) \in \hat\Z^2, 0 <q <Q}
\end{equation}
 - the points in the $\Lambda$-orbit of the point at $\infty \in \partial \half$ with denominator less than $Q$. The goal of this paper is to consider a generalization of this setup, where we replace $\Lambda$ by a general (possibly infinite covolume) discrete subgroup. For our example \eqref{Fhat} the corresponding subgroup is the Hecke group

\begin{equation}\label{Gamma_0 def}
  \wh{\Gamma} = \left<\mat{1}{4}{0}{1} , \mat{0}{1}{-1}{0} \right>.
\end{equation}

\vspace{4mm}

Most of our theorems hold for general subgroups. Hence, let $\Gamma <\PSL(2,\R)$ be a \emph{general} non-elementary, finitely generated subgroup in $G$ with critical exponent $\delta_{\Gamma}$. In our context $1/2< \delta_{\Gamma} \le 1$ and $\delta_{\Gamma}$ is equal to the Hausdorff dimension of the limit set of the subgroup (we introduce these definitions in \prettyref{sec:Background Hyp}). Furthermore assume $\Gamma$ has a cusp at $\infty$ and let $\Gamma^{\infty} = (\Gamma/\Gamma_{\infty}) \infty \subset \partial \half$ denote the orbit of $\infty$. Hence, $\Gamma^{\infty}$ is the set of the cusps located at points on the boundary, isomorphic to $\infty$. Finally we assume that $\Gamma_{\infty} = \left<\begin{psmallmatrix} 1 & 1 \\0 &1\end{psmallmatrix}\right>$. I.e that the fundamental domain is periodic with period $1$ along the real line. Note that $\wh{\Gamma}$ has period $4$. However a scaling could be applied to give it period $1$ (in order to preserve the continued fraction description we refrain from doing so). 

Let

\begin{equation} \label{eqn: primitive}
  \mathcal{Z} := \left \{ (p,q) \in (0,1) \Gamma \right\} \subset \R^2,
\end{equation}
denote the analogue of primitive vectors and define

\begin{align} \label{eqn:gen Farey}
  \begin{aligned}
    \mathcal{F}_{Q}&:= \left\{ \frac{p}{q} \in [0,1) : (p,q) \in \mathcal{Z}, 0<q<Q\right \}\\
    &= \left\{\frac{p}{q} \in \Gamma^{\infty} : 0\le p <q<Q\right\}. 
  \end{aligned}
\end{align}
$\mathcal{F}_{Q} $ is the primary object of study for this article, which we call a \emph{generalized Farey sequence} (occassionally \emph{gFs}). In Subsection \ref{subsec:Asymptotic Formula} we show that asymptotically there exists a constant $0<c_{\Gamma}<\infty$ such that

\begin{equation} \label{eqn:F_Q asymptotic}
  |\mathcal{F}_Q| \sim c_{\Gamma}Q^{2\delta_{\Gamma}}.
\end{equation}

The goal of the paper is to establish the Theorems in Sections \ref{s:Horospherical Equidistribution} - \ref{sec:Gauss-Like Measure} which we describe briefly here. As the statements of the theorems require the use of fractal measures, we present them formally only after presenting the necessary notation (readers familiar with Patterson-Sullivan theory may wish to skip ahead and see the theorems now). Sections \ref{sec:Background Hyp} and \ref{sec:Background Thms} present some background and preliminary work. Subsequently the \textbf{main results of the paper} are: 

\begin{itemize}
  \item \textbf{Counting primitive points:} In \emph{Section \ref{s:Horospherical Equidistribution}} we present a theorem for the equidistribution of the horocycle flow in infinite volume subgroups (proved by Oh and Shah \cite{OhShah2013}). Then we show how this equidistribution result can be used to prove a technical theorem about counting primitive points in a sheared set (Theorem \ref{thm:Sheared}) and another technical theorem about counting primitive points in a rotated set (Theorem \ref{thm:rotation}). These theorems generalize the analogous result for lattices in \cite{MarklofStrom2010}.
  \item \textbf{Diophantine approximation by parabolics:} We prove two theorems in metric Diophantine approximation in Fuchsian groups. These are the analogues of the Erd\"{o}s-Sz\"{u}sz-Tur\'{a}n and Kesten problems in the infinite volume setting. In the classical setting, these problems were solved using homogeneous dynamics by Marklof in {\cite[Theorem 4.4]{Marklof2000}} and Athreya and Ghosh \cite{AthreyaGhosh2015}. Moreover Xiong and Zaharescu \cite{XiongZaharescu2006} and Boca \cite{Boca2008} solved the problem using number theoretic methods (by applying the BCZ map). Extending classical results in metric Diophantine approximation to the setting of Fuchsian groups is not new and was done by Patterson \cite{Patterson1976} who proved Dirichlet and Khintchine type theorems for such parabolic points. More recently, for example Beresnevich et. al. \cite{BeresnevichETAL2018} studied the equivalent problems for Kleinian groups.

      In the same section we show that Theorem \ref{thm:rotation} allows us to prove that there is a limiting distribution for the direction of primitive points, $\cZ$, as viewed from the origin. This problem has not been addressed in the Euclidean setting except for lattices (\cite{MarklofStrom2010}).

  \item \textbf{Equidistribution of gFs:} Theorem \ref{thm:main theorem} states that the gFs equidistributes over a horospherical section.  In a series of papers (\cite{Marklof2010}, \cite{Marklof2013}), Marklof showed that the (classical) Farey sequence, when embedded into a horosphere equidistributes on a particular section. This equidistribution theorem was then used to show that the spatial statistics of the Farey sequence converge. This was followed by work of Athreya and Cheung \cite{AthreyaCheung2014} who (in dimension $d=2$) were able to construct a Poincar\'{e} section for the horocycle flow such that the return time map generates Farey points. We restrict our attention to proving the equidistribution result in this more general setting. Heersink \cite{Heersink2017} generalized \cite{Marklof2010} to certain congruence subgroups of $\Lambda$ (still in the finite covolume setting). Furthermore, the method of \cite{AthreyaCheung2014} has been generalized to more general subgroups such as Hecke triangle groups (e.g \cite{Taha2019}). However we will not discuss this approach here.

  \item \textbf{Convergence of local statistics:} Theorem \ref{thm:gap distribution}, as a consequence of Theorem \ref{thm:Sheared} and Theorem \ref{thm:main theorem}, states that two sorts of local statistics converge in the limit. A corollary of one of these is that the limiting gap distribution exists. This distribution in the classical setting was originally calculated by Hall \cite{Hall1970} (and is known as the Hall distribution) and has been studied by many people since. The Hall distribution was originally put into the context of ergodic theory in \cite{BocaCobeliZaharescu2001}.

  \item \textbf{An explicit formula for the gap distribution:} In \emph{Section \ref{sec:An Explicit Example}} we restrict to the example $\wh{\Gamma}$. For this example we show that the limiting gap distribution can be explicitly written as an integral over a compact region. While the integral involves a fractal measure this is the first time such an explicit formula has been calculated in the infinite volume setting. There is much interest in finding explicit formula for limiting gap distributions for projected lattice point sets and the infinite covolume analogue. The only instance (to our knowledge) of such explicit examples are those covered in \cite{RudnickZhang2017}. In that paper Rudnick and Zhang used the relation between Farey points and Ford circles to produce examples for which they could express the limiting gap distribution explicitly (recovering, in one instance, the Hall distribution). In Section \ref{ss:Circle Packing Example} we show that the Farey sequence for $\wh{\Gamma}$ can also be used to generate a (sparse) Ford Configuration which leads to our result.
    \item \textbf{Ergodicity of a new Gauss-like measure:} Continuing to work with the example $\wh{\Gamma}$, we show that a new fractal measure takes on the role of the Gauss measure (Theorem \ref{thm:ergodicity}).  That is,  this measure is ergodic for the Gauss map. As an application we show that this ergodicity implies convergence to an explicit function of the Gauss-Kuzmin statistics in our context. This section takes inspiration from \cite{Series1985} where Series showed how the Gauss measure can be viewed as a projection of the Haar measure on a particular cross-section.

\end{itemize}

\subsection{Ford configurations for $\wh{\Gamma}$} \label{ss:Circle Packing Example}

To give some further intuition for generalized Farey sequences, in this section we show that the gFs for $\wh{\Gamma}$ admits a simple geometric interpretation which we shall return to in section \emph{Section \ref{sec:An Explicit Example}}. Returning to our example $\wh{\cF}_Q$ - \eqref{Fhat}, note that

\begin{equation}
  \wh{\Gamma}^{\infty} = \cQ_4.
\end{equation}
To see this, simply note that the two generators in \eqref{Gamma_0 def} correspond to the maps $f(x) = x+4$ and $g(x) = \frac{-1}{x}$ which generate these continued fractions.

 Consider the action of $\wh{\Gamma}$ on an initial configuration of circles in the closure $\overline{\half}$:

 \begin{gather}
   \begin{gathered}
   \mathcal{K}_0 :=(\mathcal{C}_0,\mathcal{C}_1,\mathcal{C}_2,\mathcal{C}_3) \\
   \mathcal{C}_0 = \R \quad, \quad \mathcal{C}_1 = \R + i \quad , \quad \mathcal{C}_2 = C(i/2,1/2) \quad , \quad \mathcal{C}_3 = C(i/2+4,1/2)
   \end{gathered}
 \end{gather}
where $C(z,r)$ is a circle located at $z \in \overline{\half}$ of radius $r$. We are interested in the resulting sparse Ford configuration, $\mathcal{K}:= \wh{\Gamma} \mathcal{K}_0$, shown in Figure \ref{fig:Gamma_0}. Any group element in $\wh{\Gamma}$ can be decomposed into a composition of circle inversions through vertical lines at $0$ and $4$ and $C(0,1)$ and $C(4,1)$ (these are also shown in Figure \ref{fig:Gamma_0}).


\begin{figure}[ht!]
  \begin{center}    
    \includegraphics[width=0.9\textwidth]{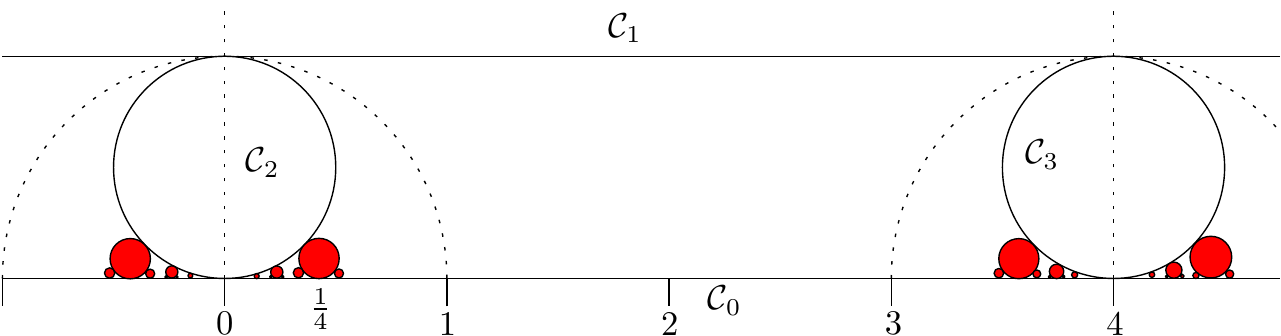}
  \end{center}
  \caption{%
     {\tt Diagram of a portion of $\mathcal{K}$. The dotted lines represent the circle inversions coresponding to the subgroup $\wh{\Gamma}$. The white circles (including the $x$-axis and horizontal line above) represent the initial configuration $\mathcal{K}_0 = (\cC_0,\cC_1,\cC_2,\cC_3)$. The filled-in circles represent some of the images. 
       }
   }%
  
  \label{fig:Gamma_0}

\end{figure}

Let $\mathcal{A}_T$ denote the set of tangencies with $\mathcal{C}_0$ in $[0,1]$ such that the circle tangent to $\mathcal{C}_0$ has diameter larger than $T^{-1}$. The way we have constructed the packing $\mathcal{K}$, these tangencies are exactly the cuspidal points of the group (i.e the tangencies are located on the orbit $\wh{\Gamma}^\infty$). Moreover one can easily show if a circle in this packing is tangent to $\mathcal{C}_0$ at $p/q$ in reduced form then the diameter is given by $1/q^2$. Hence $\cA_{Q^2} = \wh{\cF}_{Q}$, i.e the set of tangencies of circles with diameter greater than $Q^2$ is exactly the gFs of height $Q$.

Given an interval $\mathcal{I} \subset [0,1]$, let $\mathcal{A}_{T,\mathcal{I}} = \mathcal{A}_T \cap \mathcal{I}$. We label the elements of $\mathcal{A}_T = \{x^j_{T,\mathcal{I}}\}_{j=1}^{\#\mathcal{A}_{T,\mathcal{I}}}$ such that $x^j_{T, \mathcal{I}} < x^{j+1}_{T,\mathcal{I}}$ for all $j$. The gap distribution is then 

\begin{equation}\label{F0 gap}
  \wh{F}_{T,\mathcal{I}}(s) := \frac{\# \left\{ i \in [1,\#\mathcal{A}_{T,\mathcal{I}})  :  T(x^{j+1}_{T,\mathcal{I}} - x^j_{T,\mathcal{I}}) \le s \right\}  }{T^{\delta_{\wh{\Gamma}}}}
\end{equation}
for $s >0$.

\vspace{2mm}
 
In Section \ref{sec:An Explicit Example} we show that the limiting gap distribution can be explicitly calculated as a sum of integrals over compact regions involving a fractal measure presented below. This allows us to show that all gaps have size bigger than $s<2$ (not just in the limiting case), and to say something more about the regularity of $F$ and the growth of the derivative.

\begin{remark}
Of course different subgroups generate different sparse Ford configurations and have other interesting relations to continued fractions (and hence Diophantine approximation). We only address this (simplest) example here. That said, our methods generalize without additional effort to any Hecke subgroup of the form $\Gamma_c = \left< \begin{psmallmatrix}0 & c \\ 0 &  1 \end{psmallmatrix}, \begin{psmallmatrix}0 & -1 \\ 1 &  0 \end{psmallmatrix} \right>$ for $c \in \R_{>2}$ (the corresponding continued fraction description will involve $c$ rather than $4$ and this loses some elegance for non-integer $c$).
\end{remark}

  \section{Background - Hyperbolic Geometry} \label{sec:Background Hyp}

Consider the action of $G$ on $\half$ via M\"{o}bius transformations: for $z \in \half$ and $g = \begin{psmallmatrix} a & b \\ c & d \end{psmallmatrix} \in G$

\begin{align}
  \begin{aligned}
    & gz  = \frac{az +b}{cz+d} \\
    & z g = \transpose{g} z = \frac{az+c}{bz+d}.
  \end{aligned}
\end{align}
Let $X_{i} \in T^{1}(\half)$ denote the vector pointing upwards based at $i$. Denote

\begin{itemize}
  \item $K = \operatorname{Stab}_G(i)$, hence $\half \cong G/K$.
  \item $A$ - a one parameter subgroup corresponding to the unit speed geodesic flow, $\mathcal{G}_r$, on $T^{1}(\half)$. For $X_{i}$ the action of $A$ corresponds to multiplication by $\Phi^t = \begin{psmallmatrix} e^{t/2} &0 \\0 & e^{-t/2}\end{psmallmatrix}$.
  \item $N_-:= \set{n_-(x) = \mat{1}{x}{0}{1}:x \in \R }$, the contracting horosphere for $\Phi^t$.
  \item $N_+ : = \set{n_+(x) = \mat{1}{0}{x}{1}:x \in \R }$, the expanding horosphere for $\Phi^t$.
\end{itemize}
We identify points in $G$ with points in $T^1(\half)$ via the map $g \mapsto gX_{i}$ and points in $G/K$ we identify with points in $\half$ via the map $g \mapsto g i$.

\subsection{Measure Theory on Infinite Volume Hyperbolic Manifolds} \label{subsec:inf vol}

To construct the appropriate measures we require the following definitions. For a point $u \in T^1(\half)$ denote the forward and backward geodesic projections

\begin{equation}
  u^{\pm} = \lim_{r\to \infty} \mathcal{G}_r(u).
\end{equation}
Moreover, for $g\in G$ we denote $g^{\pm} = g(X_{i})^{\pm}$. Let $\cL(\Gamma) \subset \partial \half$ - the \emph{limit set} - denote the set of accumulation points of any orbit under $\Gamma$. A classical result in the field states that the Hausdorff dimension of $\cL(\Gamma)$ is the critical exponent $\delta_{\Gamma}$ (\cite{Sullivan1979}).

Given a boundary point $\xi \in \partial \half$ and two points in the interior $x,y \in \half$, define the \emph{Busemann function} to be

\begin{equation}
  \beta_{\xi} (x,y) := \lim_{t\to \infty} d(x,\xi_t) - d(y,\xi_t),
\end{equation}
where $\xi_t$ is any geodesic such that $\lim_{t\to \infty}\xi_t = \xi$. In words, the Busemann function measures the signed distance between the horospheres containing $x$ and $y$ based at $\xi$. 

Define a $\Gamma$\emph{-invariant conformal density of dimension} $\delta_{\mu}>0$ to be a family, $\{\mu_{x} : x \in \half\}$ of finite, Borel measures on the boundary $\partial \half$ such that

\begin{equation} \label{conformal invariance} 
  \gamma_{\ast}\mu_{x}(\cdot) := \mu_{x}(\gamma^{-1}\cdot) = \mu_{\gamma x}(\cdot), \quad \frac{d\mu_{x}}{d\mu_{y}}(\xi) = e^{\delta_{\mu} \beta_{\xi}(y,x)},
\end{equation}
for any $y \in \half$, $\xi\in \partial \half$ and $\gamma \in \Gamma$. Patterson \cite{Patterson1976} (in dimension $2$) and Sullivan \cite{Sullivan1979} (in higher dimensions) constructed a $\Gamma$-invariant conformal density of dimension $\delta_{\Gamma}$ supported on the limit set $\cL(\Gamma)$. We denote this conformal density $\nu_{x}$. Moreover let $\mathrm{m}_{x}$ denote the $G$-invariant density of dimension $1$ (the \emph{Lesbegue density}).

Given a point $u \in T^1(\half)$ let $\pi(u)$ denote the projection to $\half$ and let $s= \beta_{u^-}(i,\pi(u))$. From there define the following measures:

\begin{itemize}
  \item The \emph{Burger-Roblin} measure 
    \begin{equation}\label{BR}
       d\BR(u) = e^{\delta_{\Gamma}\beta_{u^-}(i,\pi(u))}e^{\beta_{u^+}(i,\pi(u))}d\nu_{i}(u^-)d\mathrm{m}_{i}(u^+)ds 
     \end{equation}
     is supported on $\{u \in T^1(\half) : u^- \in \cL(\Gamma) \}$ and is finite on $\Gamma \backslash G$ iff $\Gamma \backslash G$ has finite volume (in which case the Burger-Roblin measure is equal to the Haar measure).
  \item The \emph{Bowen-Margulis-Sullivan} measure
    \begin{equation} \label{BMS}
       d\BMS(u) = e^{\delta_{\Gamma}\beta_{u^-}(i,\pi(u))}e^{\delta_{\Gamma}\beta_{u^+}(i,\pi(u))}d\nu_{i}(u^-)d\nu_{i}(u^+)ds 
     \end{equation}
     is supported on $\{u \in T^1(\half) : u^\pm \in \cL(\Gamma) \}$ and is finite on $\Gamma \backslash G$.
\end{itemize}

Now define the \emph{Patterson-Sullivan} measure (for $N_-$) on $\partial \half \simeq \R$ to be

\begin{equation}
  d\mu^{PS}(x) := e^{\delta_{\Gamma}\beta_{x}(i, i+x)}d\nu_i(x).
\end{equation}
Note that $\supp(\mu^{PS}) = \cL(\Gamma)$. We will primarily use this Patterson-Sullivan measure, however we also use one associated to the expanding horospherical subgroup $N_+$, defined as

\begin{equation}
  d\mu^{PS}_{N_+}(x) := e^{\delta_{\Gamma}\beta_{\frac{1}{x}}(i, \frac{i}{xi+1})}d\nu_i(\frac{1}{x}).
\end{equation}

  \section{Preliminary Results} \label{sec:Background Thms}

\subsection{Proof of \eqref{eqn:F_Q asymptotic}} \label{subsec:Asymptotic Formula}

\begin{proof}[Proof of \eqref{eqn:F_Q asymptotic}]

  A rational $\frac{a}{b}$ belongs to $\cF_Q$ if and only if there exists a  $\gamma = \begin{psmallmatrix} a & \ast \\ b & \ast \end{psmallmatrix} \in \Gamma/\Gamma_{\infty}$ and $0<a<b<Q$. Using the standard Iwasawa decomposition one can write

  \begin{align}
      \gamma = \mat{\cos\theta}{-\sin\theta}{\sin\theta}{\cos\theta} \mat{y^{1/2}}{0}{0}{y^{-1/2}}
  \end{align}
  where $a= \cos\theta y^{1/2}$ and $b = \sin \theta y^{1/2}$. Therefore the problem is equivalent to counting 

  \begin{equation}
    \#\left\{\gamma \in \Gamma/ \Gamma_{\infty} : (\theta,y) \in \Omega \right \},
  \end{equation}
  where $\Omega := \{ (\theta,y) : 0 < y^{1/2} \cos\theta  <  y^{1/2}\sin \theta< Q \}$. Counting the asymptotic number of points in such a sector is the content of \cite{BourgainKontorovichSarnak2010} (see Theorem \ref{thm:bisector counting} below).
  
  Below to prove Proposition \ref{prop:nice subset} we perform this calculation more carefully (and will calculate the costant in that context, thus we leave the details till then).

\end{proof}

\subsection{Gauss-Type Decomposition} \label{subsec:NAN-Iwasawa Decomposition}

Let $M_{\vect{y}} : =\begin{psmallmatrix} y_2^{-1} & 0 \\ y_1 & y_2 \end{psmallmatrix}$, for $\vect{y} \in \R^2$. In what follows we will need the following decomposition of $T^1(\half)$.


\begin{proposition} \label{prop:Iwasawa}
  For any $\phi \in C_c(T^1(\half))$ and any set $\mathscr{A} \subset \R^{2}$

  \begin{equation}
    \int_{N_-\{M_{\vect{y}} : \vect{y} \in \mathscr{A}\}} \phi (hM_{\vect{y}})dm^{BR}(hM_{\vect{y}}) = 2 \int_{\R \times \mathscr{A} } \phi(n_-(x) M_{\vect{y}})y_2^{2\delta_{\Gamma}-2} dy_2 dy_1 d\mu^{PS}(x). \label{eqn:Iwasawa type}
  \end{equation}

\end{proposition}

\begin{proof}

  The goal is to understand the forwards and backwards orbits of $u= hM_{\vect{y}}X_{i}$. First we note that 
  \begin{equation}
    u^- = (hM_{\vect{y}} X_{i})^- = hX_{i}^- \label{eqn:h unstable}
  \end{equation}
  (this follows from the definition of the stable and unstable directions of the geodesic flow). Hence we can write:

  \begin{align}
    \begin{aligned}
      s &:= \beta_{ u^-}(i,\pi( u))\\ 
      &=  \beta_{X_{i}^-}(h^{-1}i,M_{\vect{y}} i).
    \end{aligned}
  \end{align}
  Inserting the definition of the Busemann function and using its invariance properties then gives
  \begin{align}
    \begin{aligned}
      s &=  \lim_{t\to \infty} d(h^{-1}i,\Phi^{-t}i) - d(M_{\vect{y}} i,\Phi^{-t}i)    \\
      &=  \lim_{t\to \infty} d(i,\Phi^{-t}i) - d(M_{\vect{y}} i,\Phi^{-t}i)  + d(h^{-1}i,\Phi^{-t}i) - d(i,\Phi^{-t}i).  \\
    \end{aligned}
  \end{align}
    Now setting $r_0(h) = \beta_{hX_{i}^-}(i,hi)$ gives
    \begin{align}
      \begin{aligned}
        s &=  \lim_{t\to \infty} t - d(\begin{psmallmatrix}y_2^{-1} & 0 \\ 0 &y_2 \end{psmallmatrix} i,\Phi^{-t}i)  + r_0(h)  \\
        &=  \lim_{t\to \infty} t - t + 2\ln y_2  + r_0(h)  \\
        &=   2\ln y_2  + r_0(h).  \label{eqn:y-t equality}
      \end{aligned}
  \end{align}
  Thus

  \begin{equation} \label{eqn:y-t differential}
    ds = \frac{2 dy_2}{y_2}.
  \end{equation}
  Moreover, we note that by definition

  \begin{equation} \label{PS r_0}
    e^{\delta_{\Gamma}r_0(n_-(x))}d\nu_{i}(n_-(x)X_{i}) = d\mu^{PS}(x).
  \end{equation}

  Next consider the measure

  \begin{eqnarray}
    d\lambda_{g}(z) = e^{\beta_{(hM_{\vect{y}}X_{i})^+}(i,hM_{\vect{y}} i)} dm_{i}((hM_{\vect{y}}X_{i})^+),
  \end{eqnarray}
  with $g = h \begin{psmallmatrix} y_2^{-1} & 0 \\ 0 &y_2 \end{psmallmatrix}$ and $z = n_+(y_2^{-1}y_1)$. Which we can write (using the $G$-invariance of $m$)

  \begin{align}
    \begin{aligned}
      &= e^{\beta_{(g z X_{i})^+}(i,g z i)} dm_{i}(( gz X_{i})^+) \\
      &= e^{\beta_{(g z X_{i})^+}(i,g z i)} dm_{g^{-1}i}(( z X_{i})^+) 
    \end{aligned}
  \end{align}
    and then using the definition of conformal densities:
  \begin{align}
    \begin{aligned}
      &= e^{(\beta_{(g z X_{i})^+}(i,g z i) + \beta_{(z X_{i})^+}(i,g^{-1} i) )   } dm_{i}(( z X_{i})^+) \\
      &= e^{\beta_{(zX_{i})^+}(i,z i) } dm_{i}(( z X_{i})^+). \label{eqn:N Haar calc d}
    \end{aligned}
  \end{align}
  Hence $d\lambda_g = d\lambda_{e}$ and in particular $\lambda_{e}$ is $N^+$-invariant. Hence it is the Haar measure on $N_+$. Thus we have (for $y_2$ fixed)

  \begin{equation}
    d\lambda_g(z) = dz  = y_2^{-1}dy_1 \label{eqn:Haar measure N}.
  \end{equation}
  Inserting (\ref{eqn:h unstable}), (\ref{eqn:y-t equality}), (\ref{eqn:y-t differential}), \eqref{PS r_0}, and (\ref{eqn:Haar measure N}) into the definition of the $BR$-measure we get (\ref{eqn:Iwasawa type}).

\end{proof}

\subsection{Global Measure Formula}
\label{ss:Global Measure Formula}

The last theorem from the literature we require is the so-called global measure formula stated by Stratmann and Velani {\cite[Theorem 2]{StratmannVelani1995}}, which requires some set up. In actuality we only use the simpler Corollary \ref{cor:GMF cor}. As stated in \cite{StratmannVelani1995}, there exists a disjoint, $\Gamma$-invariant collection of horoballs $\mathscr{H}$ such that $(\cC_{\Gamma}\setminus \mathscr{H})/\Gamma$ is compact, where $\cC_{\Gamma}$ is the convex hull of $\cL(\Gamma)$.

We let $\eta \in \cL(\Gamma)$ be a \emph{parabolic limit point}. Define $\eta_t$ to be the unique point along the geodesic connecting $i$ to $\eta$ whose hyperbolic distance from $i$ is $t$. And define

\begin{equation} \label{b}
  b(x) = \begin{cases} 
    0 &\mbox{ if } x \in \half\setminus \mathscr{H}\\
    d(x,\partial H_{\eta}) & \mbox{ if } x \in H_{\eta} \in \mathscr{H} 
  \end{cases},
\end{equation}
where $H_{\eta}$ is the horoball at $\eta$. 


\begin{theorem}[{\cite[Theorem 2]{StratmannVelani1995}}]\label{thm:Gmf}
  There exists a constant $0<C<\infty$ such that for any $\eta \in \cL(\Gamma)$ - a parabolic cusp and for any $t>0$,

  \begin{equation} \label{GMF}
    C^{-1}e^{-\delta_{\Gamma}t}e^{b(\eta_t)(1-\delta_{\Gamma})} \le \nu_{i}(\cB(\eta,e^{-t})) \le C e^{-\delta_{\Gamma}t}e^{b(\eta_t)(1-\delta_{\Gamma})}
  \end{equation}
  where $\cB(\eta,e^{-t}) \subset \partial \half$ is the ball centered at $\eta$ of radius $e^{-t}$
\end{theorem}


\begin{corollary}\label{cor:GMF cor}
  Assume $\eta \in \cL(\Gamma)$ is a parabolic cusp, in a small ball around $\eta$ we can approximate the measure:
  \begin{equation}
    d\nu_{i}(\eta+h) \le h^{2\delta_{\Gamma}-2}dh.
  \end{equation}
\end{corollary}

This corollary follows by differentiating \eqref{GMF} with $h=e^{-t}$ and by noting  $b(\eta_t) \le t$.

  \section{Horospherical Equidistribution} \label{s:Horospherical Equidistribution}





Consider an unstable horosphere for the geodesic flow $\Phi^t$, $N_+$. We parameterize the projection by $n_+: \T \to \Gamma \cap N_+ \backslash \Gamma N_+$. {\cite[Theorem 3.3]{Lutsko2018}} (which follows from {\cite[Theorem 3.6]{OhShah2013}}) states

\begin{theorem} \label{thm:equi x dependent}
  Let $\lambda$ be a Borel probability measure on $\T$ with continuous density with respect to Lebesgue. Then for every $f:\T \times \Gamma \backslash G \to \R$ compactly supported and continuous

  \begin{equation}
    \lim_{t\to \infty}e^{(1-\delta_{\Gamma})t}\int_{\T}f(x,n_+(x)\Phi^{t})d\lambda(x)  = \frac{1}{|m^{BMS}|} \int_{\T \times \Gamma \backslash G} f(x,\alpha)\lambda'(x) d\mu_{N_+}^{PS}(x) dm^{BR}(\alpha).
  \end{equation}

\end{theorem}

Furthermore this theorem can be applied to characteristic functions (this follows in the same way as {\cite[Corollary 3.5]{Lutsko2018}})


\begin{corollary} \label{cor:equi char funcs}
  Let $\lambda$ be a Borel probability measure on $\T$ with continuous density with respect to Lebesgue. Let $\mathcal{E} \subset \T \times \Gamma \backslash G$ be a compact set with boundary of $(\mu_{N_+}^{PS} \times m^{BR})$-measure 0. Then

  \begin{equation}
    \lim_{t\to \infty}e^{(1-\delta_{\Gamma})t}\int_{\T}\chi_{\mathcal{E}}(x,n_+(x)\Phi^{t})d\lambda(x)  = \frac{1}{|m^{BMS}|} \int_{\T \times \Gamma \backslash G} \chi_{\mathcal{E}}(x,\alpha)\lambda'(x) d\mu_{N_+}^{PS}(x) dm^{BR}(\alpha).
  \end{equation}

\end{corollary}

\subsection{Counting Primitive Points in Sheared Sets}
\label{s:Counting Primitive Points in Sheared Sets}

As a straightforward consequence of Corollary \ref{cor:equi char funcs} we have the following theorem, which (in Sections \ref{s:Diophantine} and \ref{s:Local Statistics}) we show has a number of important consequences.


\begin{theorem}\label{thm:Sheared}
  Let $\lambda$ be a Borel probabilty measure on $\T$ with continuous density with respect to Lebesgue. Let $\cA \subset \R^2$ be a compact set with boundary of Lebesgue measure $0$. Then for every $k \ge 1$:

  \begin{equation}
    \lim_{t\to \infty} e^{(1-\delta_{\Gamma})t} \lambda \left(\{ x \in \T : \left|\cZ n_+(x)\Phi^t \cap \cA )\right| = k \}\right) = \frac{C_{\lambda}}{|\BMS|}\BR(\{\alpha \in \Gamma \backslash G : | \cZ \alpha \cap \cA| = k \}),
  \end{equation}
  where $C_{\lambda} = \mu^{PS}_{N_+}(\lambda')$.

\end{theorem}

Theorem \ref{thm:Sheared} is an infinite covolume version of {\cite[Theorem 6.7]{MarklofStrom2010}}. The proof is a straightforward consequence of Corollary \ref{cor:equi char funcs} and the fact that if $\cA$ is compact and has boundary of Lebesgue measure $0$, then 

\begin{equation}
  \left\{ g \in \Gamma \backslash G : \cZ g \cap \cA = k \right \}
\end{equation}
is compact and has boundary of volume $0$, and the Burger-Roblin measure of a $0$ volume set is $0$.

Using {\cite[Theorem 6.10]{MohammadiOh2015}} in the same way we used {\cite[Theorem 3.6]{OhShah2013}} to derive Theorem \ref{thm:equi x dependent}, we have


\begin{theorem} \label{thm:Sheared PS}
  Let $\cA \subset \R^2$ be a compact set with boundary of Lebesgue measure $0$. Then for every $k \ge 1$:

  \begin{equation}
    \lim_{t\to \infty} \mu^{PS}_{N_+} \left(\{ x \in \T : \left|\cZ n_+(x)\Phi^t \cap \cA )\right| = k \}\right) = \frac{|\mu_{N_+}^{PS}|}{|\BMS|}\BMS(\{\alpha \in \Gamma \backslash G : | \cZ \alpha \cap \cA| = k \}).
  \end{equation}

\end{theorem}

In words each of these two theorems is asking for the limiting probability that a randomly sheared set contains $k$ points. In one instance (Theorem \ref{thm:Sheared}) we randomly shear the set with measure $\lambda$ and in the other (Theorem \ref{thm:Sheared PS}) we use the measure $\mu_{N_+}^{PS}$.

\subsection{Counting Primitive Points in Rotated Sets}

Similarly to Section \ref{s:Counting Primitive Points in Sheared Sets} one can ask about the probability of finding $k$ primitive points in a randomly rotated set (as oppose to a randomly sheared one). In {\cite[Section 6]{Lutsko2018}} we show that similar equidistribution results to Theorem \ref{thm:equi x dependent} and Corollary \ref{cor:equi char funcs} also hold when the horospherical subgroup $N_+$ is replaced with the rotational subgroup, $K$. Parameterize the rotation subgroup $K$ by the boundary $\partial \half$ in the natural way $x \mapsto R(x)$. Then the rotational Patterson-Sullivan measure is defined to be

\begin{equation}
  d\mu^{PS}_K(x) = e^{\beta_x(i, R(x) (e i))}d\nu_i(x).
\end{equation}
Note $\mu^{PS}_{K}$ is supported on  $\cL(\Gamma)$. Hence, the analogous theorem to Theorem \ref{thm:Sheared} follows from {\cite[Corollary 6.2]{Lutsko2018}} (in exactly the same way that Theorem \ref{thm:Sheared} follows from Corollary \ref{cor:equi char funcs}):


\begin{theorem}\label{thm:rotation}
  Let $\lambda$ be a Borel probability measure on $\T$ with continuous density with respect to Lebesgue. Let $\cA \subset \R^2$ be a compact subset with boundary of Lebesgue measure $0$. Then for every $k \ge 1$

  \begin{equation}
    \lim_{t\to \infty} e^{(1-\delta_{\Gamma})t}\lambda\left(\{x \in \T : \left| \cZ R(x) \Phi^t \cap \cA \right|  = k \}\right) = \frac{D_{\lambda}}{|\BMS|} \BR (\{ \alpha \in \Gamma \backslash G : |\cZ \alpha \cap \cA | = k \})
  \end{equation}
  where $D_{\lambda} = \mu^{PS}_{K}(\lambda')$.
  
\end{theorem}






  \section{Consequences of Theorems \ref{thm:Sheared} and \ref{thm:rotation}}
\label{s:Diophantine}

\subsection{Diophantine Approximation in Fuchsian Groups}
\label{ss:Dio}

Theorem \ref{thm:Sheared} can be used to prove several statements about the set of numbers which can be approximated by parabolic points in the limit set of the Fuchsian groups studied here. In particular, as discussed in \cite{AthreyaGhosh2015}, Erd\"{o}s-Sz\"{u}sz-Tur\'{a}n (henceforth abreviated EST) introduced the following problem in Diophantine approximation: what is the probability that a uniformly chosen point, $x \in [0,1]$, satisfies

\begin{equation}\label{mod Dirichlet}
  \left| x - \frac{p}{q} \right| \le \frac{A}{q^2}
\end{equation}
for $\frac{p}{q} \in \mathbb{Q}$ with $q \in [\theta Q, Q]$ for a fixed triple $(A,\theta,Q) \in \R_{>0}\times (0,1)\times \R_{>0}$? Hence if we let $EST(A,\theta,Q)$ be the random variable: the number of solutions to \eqref{mod Dirichlet}, the EST problem is to prove the existence of 

\begin{equation}
  \lim_{Q \to \infty} \mathbb{P}(EST(A,\theta,Q)>0).
\end{equation}
The limiting distribution for this random variable is given in \cite{AthreyaGhosh2015} in great generality. Our goal in this section is to understand the same problem with the rationals replaced by $\Gamma^{\infty}$.

Given a triple $(A,\theta,Q)$ as above and a number $x$, define (the analogue of the random variable $EST$), $E(A,\theta,Q)$ to be the number of solutions, $(p,q) \in \mathcal{Z}$, to

\begin{equation} \label{EST}
  | p - q x | \le \frac{A}{q}.
\end{equation}


\begin{theorem} \label{thm:EST}
  Given $(A,\theta) \in \R_{>0}\times (0,1) $. Let $\lambda$ be a Borel probability measure on $[0,1)$, with continuous density with respect to Lebesgue. Then

  \begin{equation} \label{EST lim}
    \lim_{Q\to \infty} Q^{2(1-\delta_{\Gamma})} \lambda (\{x \in [0,1) : E(A,\theta,Q) = k\}) = \frac{C_{\lambda}}{|\BMS|} \BR (\{ \alpha \in \Gamma \backslash G : | \mathcal{Z} \alpha \cap \mathfrak{C}_{A,\theta}| = k \}),
  \end{equation}
  where 

  \begin{equation}
    \mathfrak{C}_{A,\theta} := \{(x_1,x_2) \in \R \times \R : |x_1| x_2 \le A : \theta <x_2 < 1\}.
  \end{equation}
  Moreover,

  \begin{equation} \label{EST PS}
    \lim_{Q\to \infty} \mu_{N_+}^{PS} (\{x \in \cL(\Gamma) \cap [0,1) : E(A,\theta,Q)=k\}) =  \frac{1}{|\BMS|} \BMS (\{ \alpha \in \Gamma \backslash G : | \mathcal{Z} \alpha \cap \mathfrak{C}_{A,\theta}| = k \}).
  \end{equation}

\end{theorem}

\begin{proof}

  Write \eqref{EST lim} as (with $Q=e^{t/2}$)

  \begin{multline}
    \lim_{t \to \infty}  e^{(1-\delta_{\Gamma})t} \lambda \left ( \set{ x \in [0,1] : \#\left\{ (p,q) \in \cZ : (p,q) \mat{1}{0}{-x}{1} \mat{Q}{0}{0}{Q^{-1}} \in \mathfrak{C}_{A,\theta} \right\} = k }\right)\\
      = \lim_{t \to \infty} e^{(1-\delta_{\Gamma})t} \lambda \left ( \set{ x \in [0,1] : \#\left(  \cZ n_+(-x) \Phi^t \cap \mathfrak{C}_{A,\theta} \right) = k }\right).
  \end{multline}
  To which we apply Theorem \ref{thm:Sheared} to get \eqref{EST lim}.

  \eqref{EST PS} follows in the same way except, in the last step, we apply Theorem \ref{thm:Sheared PS} instead of Theorem \ref{thm:Sheared}.

\end{proof}

Moreover, the same proof allows one to prove the \emph{Kesten problem} in our context, stated as follows: for $A>0$ and $Q$ fixed let $K(A,Q)$ denote the number of solutions to 

\begin{equation} \label{Kesten}
  \abs{\alpha q - p } \le \frac{A}{Q} \quad , \quad 1 \le q \le Q.
\end{equation}
In this case the following theorem holds:


\begin{theorem}\label{thm:Kesten}
  Given $A>0$ Theorem \ref{thm:EST} holds with $E(A,\theta,Q)$ replaced by $K(A,Q)$ and $\mathfrak{C}_{A,\theta}$ replaced by

  \begin{equation}
    R_A = \set{(x,y) \in \R^2 : \abs{x} \le A, 0 \le y \le 1 }
  \end{equation}

\end{theorem}

\subsection{Directions of Primitive Points}

Given a point in $\R^2$ (taken here to be the origin, however this is not necessary), one can ask how the directions of primitive points $\cZ$ distribute for an observer at that point. The corollary of Theorem \ref{thm:rotation} below answers this question.

Let $\cD_t(\sigma, v) \subset S^1_1$ be the interval in the unit sphere with center $v$ and length $\sigma e^{-t}$, and set

\begin{equation}
  \cN_t(\sigma, v;\cZ) := \#\left\{ \vect{y} \in \cZ_t : \|\vect{y}\|^{-1} \vect{y} \in \cD_t(\sigma,v) \right\},
\end{equation}
where $\cZ_t = \{z \in \cZ : \|z\|\le e^t\}$. 


\begin{corollary}\label{cor:directions}

  Let $\lambda$ be a probability measure on $\T$, with continuous density with respect to Lebesgue. For $k \in \N_{>0}$ we have

  \begin{equation}
    \lim_{t\to \infty} e^{(1-\delta_{\Gamma})t}\lambda \left( \{ v\in \T : \cN_t(\sigma, v;\cZ) = k \}\right) = \frac{D_{\lambda}}{|\BMS|}\BR\left(\{ \alpha \in \Gamma \backslash G : |\cZ \alpha \cap \mathfrak{C}_\sigma| = k \}\right)
  \end{equation}
  where, in polar coordinates

  \begin{equation}
    \mathfrak{C}_{\sigma} = \{ x =(r\theta) \in \R^2 : r< 1 , |\theta| < \sigma \pi \}.
  \end{equation}

\end{corollary}
\noindent This Corollary follows directly from Theorem \ref{thm:rotation}.

  \section{Equidistribution of gFs} \label{sec:Proof of Main Theorem}

\subsection{Statement} \label{subsec:equidistribution}

In addition to Theorem \ref{thm:Sheared} another important consequence of the equidistribution statements in Section \ref{s:Horospherical Equidistribution}, is the following theorem, stating that the gFs equidistributes on a horospherical section. This is a generalization of {\cite[Theorem 6]{Marklof2010}}, to the infinite covolume setting.


\begin{theorem}\label{thm:main theorem} 
  Let $\sigma \in \R$ and $Q=e^{(t-\sigma)/2}$. Let $f: \T \times \Gamma \backslash G  \to \R$ be bounded continuous and supported on a connected set with finite volume. Then \small

  \begin{equation}\label{eqn:equidistribution}
    \lim_{t \to \infty} e^{-\delta_{\Gamma}t} \sum_{r \in \mathcal{F}_Q} f(r,n_-(r)\Phi^{-t})     =\frac{e^{(\delta_{\Gamma}-1)\sigma}}{|m^{BMS}|}\int_{\T\times \T}\int_{\sigma}^{\infty}\tilde{f}(x,n_-(w)\Phi^{-r})e^{\delta_{\Gamma}r}drd\mu^{PS}(w) d\mu_{N_+}^{PS}(x)
  \end{equation}\normalsize
  where $\tilde{f}(x,\alpha) := f(x, \transpose{\alpha}^{-1})$ and $r_0(h)$ is defined in Proposition \ref{prop:Iwasawa}.
\end{theorem}

\begin{remark}
  \cite{Marklof2010} and \cite{Marklof2013} treat Farey sequences in general dimension. However in the infinite covolume setting equidistribution results for $\SL(d,\R)$ have not yet been proved (to our knowledge). 
\end{remark}

\subsection{Proof}

\begin{proof}[Proof of \prettyref{thm:main theorem}]

The proof will follow the same lines as {\cite[Proof of Theorem 6]{Marklof2010}} with several exceptions as we are not working with Haar measure.

Note first that by setting $f(x,\alpha) = f_0(x,\alpha \Phi^{-\sigma})$ for $f_0$ bounded and continuous we may assume that $\sigma=0$. \\

\begin{steps}[leftmargin=0cm,itemindent=.5cm,labelwidth=\itemindent,labelsep=0cm,align=left]


  \item \label{itm:compact support} \hspace{2mm} First we show that we can reduce the theorem to $f$ compactly supported via a standard approximation argument. Assume the theorem holds for compactly supported functions. Now consider a bounded, continuous function, $f$ supported on a finite-volume set. Fix $\epsilon>0$ and consider (for some $t$) the difference

\begin{equation} \label{f compact}
  \left| e^{-\delta_{\Gamma}t} \sum_{r \in \mathcal{F}_Q} f(n_-(r) \Phi^{-t})) - \frac{1}{|m^{BMS}|} \int_{\T} \int_0^{\infty} \tilde{f}(n_-(w)\Phi^{-r})e^{\delta_{\Gamma} r}dr d\mu^{PS}(w)  \right|.
\end{equation}
Now decompose $f = f_1 + f_2$ such that $f_1$ is supported on a compact set and $f_2$ is supported on a set of volume $\varrho>0$ (as $\mbox{supp}(f)$ has finite volume $\varrho$ can be chosen arbitrarily small) and both are bounded and continuous. Hence the difference \eqref{f compact} is bounded above by

\begin{multline}
  \left|e^{-\delta_{\Gamma}t} \sum_{r \in \mathcal{F}_Q} f_1(n_-(r) \Phi^{-t})) - \frac{1}{|m^{BMS}|} \int_{\T} \int_0^{\infty}\tilde{f_1}(n_-(w)\Phi^{-r})e^{\delta_{\Gamma} r}dr d\mu^{PS}(w)  \right|\\
  + \left|e^{-\delta_{\Gamma}t} \sum_{r \in \mathcal{F}_Q} f_2(n_-(r) \Phi^{-t})) - \frac{1}{|m^{BMS}|} \int_{\T} \int_0^{\infty} \tilde{f_2}(n_-(w)\Phi^{-r})e^{\delta_{\Gamma} r}dr d\mu^{PS}(w)  \right|.
\end{multline}
Applying \prettyref{thm:main theorem} for compact functions implies we can take $t$ large enough that the first term is less than $\epsilon/2$. Since $f$ is supported on a finite connected set and $f$ is bounded, then it follows from {\cite[Theorem 4.2]{Lutsko2018}} that there exists a $C>0$ such that
\begin{equation}
  \left|e^{-\delta_{\Gamma}t} \sum_{r \in \mathcal{F}_Q} f_2(n_-(r) \Phi^{-t}))\right| \le C \int_{\supp{f_2}} f_2(g) d\BR(g) \le \epsilon/4
\end{equation}
for all $t>t_0$ and some $t_0>0$. Lastly, consider

\begin{equation}
  \left|\int_{\T} \int_0^{\infty} \tilde{f_2}(n_-(w)\Phi^{-r})e^{\delta_{\Gamma} r}dr d\mu^{PS}(w))  \right| < \infty.
\end{equation}
As $\Gamma$ has a cusp, $\delta_{\Gamma}>1/2$ which is greater than $0$. Thus the Patterson-Sullivan measure of $\mbox{supp}(\tilde{f}_2) \cap \cL(\Gamma)$ goes to $0$ as $\vol(\mbox{supp}(\tilde{f}_2))$ goes to $0$. Hence we can choose $\varrho$ such that \eqref{f compact} is bounded by $\epsilon$. Thus \prettyref{thm:main theorem} for compactly supported $f$ implies the theorem for $f$ with finite volume support.

Henceforth take $f$ to be compactly supported.
 
 
  \item \label{itm:uniform continuity} \hspace{2mm} Note that because $f$ is continuous and has compact support it is uniformly continuous. Hence for every $\varrho>0$ there exists a $\epsilon>0$ such that for all $(x,\alpha), (x'\alpha') \in \R \times G$

    \begin{equation}
      |x-x'|<\epsilon \;\;\;\;\;\;\; d(\alpha,\alpha')<\epsilon
    \end{equation}
    imply $|f(x,\alpha)-f(x',\alpha')|<\varrho$


  \item \label{itm:setup} \hspace{2mm} For $0\le \theta <1$ and $\epsilon>0$ define
    \begin{eqnarray}
      \mathcal{F}_{Q,\theta} &:=& \left\{ \frac{p}{q} \in [0,1): (p,q) \in \mathcal{Z}, \; \theta Q<q<Q\right \}\\
      \mathcal{F}^{\epsilon}_{Q} &:=& \bigcup_{r\in \mathcal{F}_{Q,\theta}+\Z}\left\{ x\in \R : \| x-r\|<\epsilon e^{-t} \right\}.
    \end{eqnarray}
    The latter we can write as
    
    \begin{equation}
      \mathcal{F}_{Q}^{\epsilon} = \bigcup_{\vect{a} \in \mathcal{Z}} \left\{ x\in \R: (a_1,a_2)n_{+}(x)\Phi^t \in \mathfrak{C}_{\epsilon}\right \},
    \end{equation}
    where

    \begin{equation} \notag
      \mathfrak{C}_{\epsilon} := \{(y_1,y_2) \in \R^2 : |y_1| < \epsilon y_2,\quad \theta<y_2\le 1 \}.
    \end{equation}

    Our goal is to write the characteristic function for $\mathcal{F}_{Q}^{\epsilon}$ as a sum over simpler characteristic functions which can write as a disjoint union. Thus, let

    \begin{equation}
      \mathcal{H}_{\epsilon}:= \bigcup_{\vect{a}\in \mathcal{Z}} \mathcal{H}_{\epsilon}(\vect{a}),   
      \quad \quad
      \mathcal{H}_{\epsilon}(\vect{a}):= \{\alpha \in G : (a_1,a_2) \alpha \in \mathfrak{C}_{\epsilon} \}.
    \end{equation}
    By considering the bijection
    
    \begin{equation}\notag
      \Gamma_{N_-} \backslash \Gamma \to \mathcal{Z}, 
      \quad \quad 
      \Gamma_{N_-}\gamma \mapsto (0,1) \gamma
    \end{equation}
    we can write

    \begin{align}
      \begin{aligned}
        \mathcal{H}_{\epsilon} &= \bigcup_{\gamma \in \Gamma_{N_-} \backslash \Gamma} \mathcal{H}_{\epsilon}((0,1) \gamma)\\
      &=\bigcup_{\gamma \in \Gamma_{N_-} \backslash \Gamma} \gamma \mathcal{H}_{\epsilon}^1,
      \end{aligned}
   \end{align}
    where
    \begin{equation} \notag
      \mathcal{H}_{\epsilon}^1 := \mathcal{H}_{\epsilon}((0,1)) = H \{M_{\vect{y}} : \vect{y} \in \mathfrak{C}_{\epsilon}\}
    \end{equation}
    with $M_{\vect{y}} := \mat{y_2^{-1}}{0}{y_1}{y_2}$.


  \item \label{itm:Disjointedness} \hspace{2mm}

    \emph{Claim:} 
    Given  $\mathcal{C} \subset G$ compact there exists an $\epsilon_0>0$ such that for all $\epsilon < \epsilon_0$

    \begin{equation} \label{eqn:disjointedness}
      \gamma \mathcal{H}^1_{\epsilon} \cap  \mathcal{H}_{\epsilon}^1 \cap \Gamma \mathcal{C} = \emptyset,
    \end{equation}
      for all $\gamma \in \Gamma / \Gamma_{N_-} \neq 1$

      \begin{proof}[Proof of Claim] 
        (\ref{eqn:disjointedness}) is equivalent to
        \begin{equation}
          \mathcal{H}_{\epsilon}((p,q)) \cap \mathcal{H}_{\epsilon}^1 \cap \Gamma \mathcal{C} = \emptyset, \quad  \quad \forall (p,q)\neq (0,1) \in \mathcal{Z}
        \end{equation}
        
        Consider an $\alpha \in G$ such that $(p,q)\alpha \in \mathfrak{C}_{\epsilon}$. We can write any such $\alpha$ as 

          \begin{equation}
            \alpha = \mat{1}{b}{0}{1} \mat{y_2^{-1}}{0}{y_1}{y_2}
          \end{equation}
          for $b\in \R$ and $y_1 \in \R$.
          
          Therefore if we assume for the sake of contradiction that $(p,q) \alpha \in \mathfrak{C}_{\epsilon}$ and $(0,1)\alpha \in \mathfrak{C}_{\epsilon}$ we have the following 4 inequalities

          \begin{equation}
            |y_2^{-1}p+(pb+q)y_1| < \epsilon y_2 (pb+q) \label{eqn:ineq 1}
          \end{equation}
          \begin{equation}
            \theta < y_2(pb+q) \le 1 \label{eqn:ineq 2}
          \end{equation}
          \begin{equation}
            |y_1| < \epsilon y_2 \label{eqn:ineq 3}
          \end{equation}
          \begin{equation}
            \theta <y_2 \le 1 \label{eqn:ineq 4}.
          \end{equation}
          Using (\ref{eqn:ineq 1}) and (\ref{eqn:ineq 4}) gives that

          \begin{equation}
            |(pb+q)y_1| < \epsilon
          \end{equation}
          which (plugging back into (\ref{eqn:ineq 4})) gives

          \begin{equation}
            |y_2^{-1}p| <2\epsilon.
          \end{equation}
          Hence 

          \begin{equation}
            |p | < 2 \epsilon.
          \end{equation}
          Thus $p = 0$. Therefore $(0,q) = (0,1)\gamma$ for some $\gamma \in \Gamma$. However since $\Gamma_{\infty} =  \left<  \begin{psmallmatrix} 1 & 1 \\ 0 & 1 \end{psmallmatrix} \right>$, $q =1$. Which is a contradiction proving the statement.
      \end{proof}


  \item \label{itm:Applying} \hspace{2mm} The claim implies that for $\mathcal{C} \subset G$ compact there is an $\epsilon_0>0$ such that for all $\epsilon< \epsilon_0$ such that

    \begin{equation} \label{eqn:dis-union}
      \mathcal{H}_{\epsilon} \cap \Gamma \mathcal{C} = \bigcup_{\gamma \in \Gamma / \Gamma_{N_-}} (\gamma \mathcal{H}_{\epsilon}^1 \cap \Gamma \mathcal{C})
    \end{equation}
    is a disjoint union. Thus let $\chi_{\epsilon}$ and $\chi^1_{\epsilon}$ denote the characteristic functions of $\mathcal{H}_{\epsilon}$ and $\mathcal{H}_{\epsilon}^1$ respectively, then

    \begin{equation}
      \chi_{\epsilon}(\alpha) = \sum_{\gamma \in \Gamma_{N_-} \backslash \Gamma} \chi_{\epsilon}^1(\gamma \alpha)
    \end{equation}
    for all $\alpha \in \Gamma \mathcal{C}$. Moreover all of the sets we consider have boundary of $BR$-measure $0$. Set $\tilde{\chi}_{\epsilon}(\alpha) := \chi_{\epsilon}(\transpose{\alpha}^{-1})$ and note that $\chi_{\epsilon}(n_+(x) \Phi^{t}) = \tilde{\chi}_{\epsilon}(n_-(-x) \Phi^{-t})$ is the characteristic function for $\mathcal{F}_Q^{\epsilon}$.

    Therefore we write

    \begin{align}\label{eqn:F-Q-epsilon integral}
      \begin{aligned}
      \int_{\mathcal{F}_Q^{\epsilon} / \Z} f(x, n_-(x) \Phi^{-t})dx &= \int_{\T}f(x,n_-(x)\Phi^{-t})\chi_{\epsilon}(n_+(-x)\Phi^{t})dx\\
      &= \int_{\T} \tilde{f}(x,n_+(-x)\Phi^{t})\chi_{\epsilon}(n_+(-x)\Phi^{t})dx,
      \end{aligned}
    \end{align}
    to which we can apply Theorem \ref{thm:equi x dependent} giving:

    \begin{equation} 
      \lim_{t\to \infty} e^{(1-\delta_{\Gamma})t} \int_{\mathcal{F}_Q^{\epsilon} / \Z} f(x, n_-(x) \Phi^{-t})dx =  \frac{1}{|m^{BMS}|} \int_{\Gamma \backslash G \times \T}\tilde{f}(x,\alpha) \chi_{\epsilon}(\alpha) dm^{BR}(\alpha) d\mu_{N_+}^{PS}(x).
    \end{equation}

    Which we write this

    \begin{align}
      \begin{aligned}\label{eqn:F-Q-epsilon equi}
       &=  \frac{1}{|m^{BMS}|} \int_{\Gamma_{N_-} \backslash G \times \T}\tilde{f}(x,\alpha) \chi^1_{\epsilon}(\alpha) dm^{BR}(\alpha) d\mu_{N_+}^{PS}(x), \\
       &=  \frac{1}{|m^{BMS}|} \int_{\Gamma_{N_-} \backslash N_-\{M_{\vect{y}}:\vect{y} \in \mathfrak{C}_{\epsilon}\} \times \T}\tilde{f}(x,\alpha) dm^{BR}(\alpha) d\mu_{N_+}^{PS}(x). 
       \end{aligned}
    \end{align}


  \item \label{itm:Calculation} \hspace{2mm} 
    
    Using Proposition \prettyref{prop:Iwasawa} we write (\ref{eqn:F-Q-epsilon equi}) as (noting that $(0,1)n_- =(0,1)$)

    \begin{equation}
      = \frac{2}{|m^{BMS}|} \int_{\T \times \{\vect{y} \in \mathfrak{C}_{\epsilon}\} \times \T} y_2^{2\delta_{\Gamma}-2}\tilde{f}(x,n_-(w)M_{\vect{y}})  dy_2 dy_1 d\mu^{PS}(w) d\mu_{N_+}^{PS}(x). 
    \end{equation}
    Which we can write
    \begin{equation}      
      = \frac{2}{|m^{BMS}|} \int_{\T \times \T}\int_{\theta}^1\int_{\mathcal{B}_{\epsilon y_2}(0)} y_2^{2\delta_{\Gamma}-2}\tilde{f}(x,n_-(w)M_{\vect{y}})  dy_2 dy_1 d\mu^{PS}(w) d\mu_{N_+}^{PS}(x). 
    \end{equation}

    Next we write $D(y_2) := \mat{y_2^{-1}}{0}{0}{y_2}$ and note

    \begin{equation}
      d(M_{\vect{y}},D(y_2)) = d(n_+(y_2^{-1}y_1), Id) \le \epsilon
    \end{equation}
    for $\vect{y} \in \mathfrak{C}_{\epsilon}$ (this is the same calculation as {\cite[(3.42)]{Marklof2010}}). Therefore

    \begin{align}
      \begin{aligned} \label{eqn:F-Q-epsilon equi 3}
      &\left| \mbox{(\ref{eqn:F-Q-epsilon equi})} - \frac{2}{|m^{BMS}|} \int_{\T \times \T}\int_{\theta}^1\int_{\mathcal{B}(\epsilon y_2)} \tilde{f}(x, n_-(w)D(y_2))y_2^{2\delta_{\Gamma}-2}  dy_2 dy_1 d\mu^{PS}(w) d\mu_{N_+}^{PS}(x)\right|\\
      & \qquad =\left| \mbox{(\ref{eqn:F-Q-epsilon equi})} - \frac{4\epsilon}{|m^{BMS}|} \int_{\T \times \T}\int_{\theta}^1 \tilde{f}(x, n_-(w)D(y_2))y_2^{2\delta_{\Gamma}-1}  dy_2 d\mu^{PS}(w) d\mu_{N_+}^{PS}(x)\right| \\
      & \qquad \le \frac{4\varrho\epsilon|\mu^{PS}|^2}{|m^{BMS}|} \int_{\theta}^1 y_2^{2\delta_{\Gamma}-1}dy_2.
    \end{aligned}
  \end{align}

  Evaluating this integral then gives that (\ref{eqn:F-Q-epsilon equi 3}) is equal to

  \begin{equation}
    \frac{2\epsilon\varrho|\mu^{PS}|^2}{|m^{BMS}|\delta_{\Gamma}} (1-\theta^{2\delta}).
  \end{equation}

  Turning now to the right term in the modulus on the left hand side (\ref{eqn:F-Q-epsilon equi 3}) note

  \begin{equation}
    \frac{4\epsilon}{|m^{BMS}|} \int_{\T \times \T}\int_{\theta}^1 \tilde{f}(x,n_-(w)D(y_2))  y^{2\delta_{\Gamma}-1} dy_2 d\mu^{PS}(w)d\mu^{PS}_{N_{+}}(x)
  \end{equation}

  Performing a final change of variables and writing $y_2 = e^{r/2}$, we conclude that

  \begin{multline}\label{eqn:F-Q-epsilon-equi-conc}
    \left| \lim_{t\to \infty}e^{(1-\delta_{\Gamma})t}\int_{\mathcal{F}_Q^{\epsilon} / \Z} f(x, n_-(x) \Phi^{-t})dx\right.\\
    \left.- \frac{2\epsilon}{|m^{BMS}|}\int_{\T \times \T}\int_0^{2|\ln\theta|}\tilde{f}(x,n_-(w)\Phi^{-t})e^{\delta_{\Gamma}r}dr d\mu^{PS}(h) d\mu_{N_+}^{PS}(x)\right|\\    <\frac{2 \varrho\epsilon |\mu^{PS}|^2}{|m^{BMS}|\delta_{\Gamma}}(1-\theta^{2\delta_{\Gamma}}).
  \end{multline}


  \item \label{itm:Conclusion} \hspace{2mm} 

    To conclude consider
    
    \begin{equation}
      \lim_{t\to \infty} e^{-\delta_{\Gamma}t } \sum_{r \in \mathcal{F}_{Q,\theta}} f(r, n_-(r) \Phi^{-t})
    \end{equation}
    taking the asymptotic formula (\ref{eqn:F_Q asymptotic}) and using  a volume estimate together with uniform continuity (see {\cite[(3.49)]{Marklof2010}} for details) we can write this as

    \begin{equation}
      = \lim_{\epsilon \to 0}\lim_{t\to \infty}  \frac{e^{(1-\delta_{\Gamma})t}}{e^{t}} \frac{e^{t}}{2\epsilon}\sum_{r\in \mathcal{F}_{\theta, Q}}\int_{|x-r|< \epsilon e^{-t}} f(x,n_-(x)\Phi^{-t}) dx.
    \end{equation}
     Which is equal

    \begin{equation}
      = \lim_{\epsilon\to 0}\lim_{t\to \infty} \frac{e^{(1-\delta_{\Gamma})t}}{2\epsilon} \sum_{r\in \mathcal{F}_{\theta, Q}}\int_{|x-r|< \epsilon e^{-t}} f(x,n_-(x)\Phi^{-t}) dx
    \end{equation}
    Then using the disjoint union in (\ref{eqn:dis-union}) we can say

    \begin{equation}
      = \lim_{\epsilon\to 0}\lim_{t\to \infty} \frac{e^{(1-\delta_{\Gamma})t}}{2\epsilon}  \int_{\mathcal{F}_{Q}^{\epsilon}\backslash \Z} f(x,n_-(x)\Phi^{-t}) dx
    \end{equation}
    and using (\ref{eqn:F-Q-epsilon-equi-conc}) we thus conclude after taking $\epsilon \to 0$ (and therefore $\varrho \to 0$) this is equal

    \begin{equation}
      =\frac{1}{|m^{BMS}|}\int_{\T\times \T}\int_0^{2|\ln\theta|}\tilde{f}(x,n_-(w)\Phi^{-r})e^{\delta_{\Gamma}r}drd\mu^{PS}(w) d\mu_{N_+}^{PS}(x)
    \end{equation}
    Taking the limit as $\theta \to 0$ is then possible as

    \begin{equation}
      \limsup_{t\to \infty} \frac{|\mathcal{F}_Q \backslash \mathcal{F}_{Q\theta}|}{e^{\delta_{\Gamma}t}} = \theta c_{\Gamma}
    \end{equation}

\end{steps}

\end{proof}

  \section{Local Statistics} \label{s:Local Statistics}

Theorem \ref{thm:Sheared} and Theorem \ref{thm:main theorem} can also be used to study the local statistics of $\cF_Q$ when viewed as a point process on $[0,1]$ (note once more we are assuming for notation, that $\Gamma^{\infty}$ is periodic on $[0,1]$).

\subsection{Statement} \label{ss:Statement Gap}

For $Q = e^{t/2}$. Let $\mathscr{A} \subset \R$ be bounded interval and set $\mathscr{A}_t = \mathscr{A}e^{-t}$. For a bounded $\mathcal{D} \subset \T$, define

\begin{equation}
  P_Q(\mathcal{D},\mathscr{A},k) = \frac{e^{t} \vol(\{x \in \mathcal{D} : |x +\mathscr{A}_t + \Z \cap \mathcal{F}_Q| = k\})}{\mu_{N_+}^{PS}(\mathcal{D})e^{\delta_{\Gamma}t}}
\end{equation}
and 

\begin{equation}
  P_{0,Q}(\mathcal{D},\mathscr{A},k) = \frac{ |\{r \in \mathcal{F}_Q \cap \mathcal{D} : |r +\mathscr{A}_t + \Z \cap \mathcal{F}_Q| = k\})}{\mu_{N_+}^{PS}(\mathcal{D})e^{\delta_{\Gamma}t}}
\end{equation}


\begin{theorem} \label{thm:gap distribution}
  Given an interval $\mathscr{A}\subset \R$ and $\mathcal{D} \subset \T$ then for all $k>0$

  \begin{equation} \label{eqn:first stats}
    \lim_{Q \to \infty} P_Q(k,\mathcal{D},\mathscr{A}) = P(k,\mathscr{A})
  \end{equation}

  \begin{equation} \label{eqn:second stats}
    \lim_{Q \to \infty} P_{0,Q}(\mathcal{D},\mathscr{A},k) = P_0(k,\mathscr{A})
  \end{equation}
  where $P(k,\mathscr{A})$ and $P_0(k,\mathscr{A})$ are given explicitly.
\end{theorem}

\begin{remark}
  In particular \eqref{eqn:second stats} implies that the limiting gap distribution exists everywhere.
\end{remark}

\begin{remark}
  Note that the above theorem is restricted to $k>0$. The reason for this is that the scaling in $P_Q$ and $P_{0,Q}$ is incorrect for the case $k=0$. For geometrically finite subgroups the boundary points cluster close together in far apart cluster. This phenomenon was noticed by Zhang \cite{Zhang2017} and again in \cite{Lutsko2018}.

\end{remark}

To give another qualitative example, we have graphed the gap distribution for $\wh{\Gamma}^{\infty}$ in Figure \ref{fig:gap}. 


\begin{figure}[ht!]
  \begin{center}    
    \includegraphics[width=0.9\textwidth]{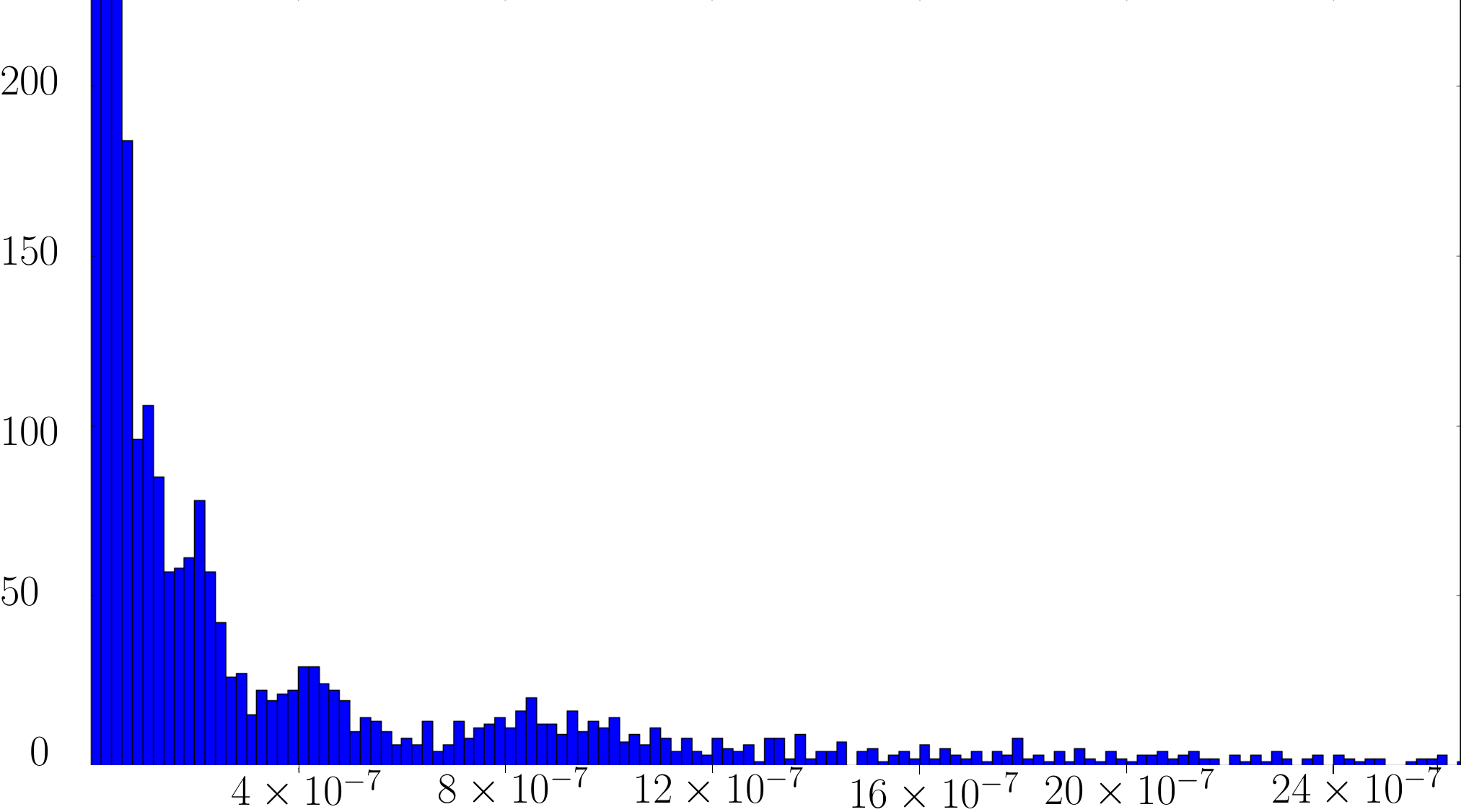}
  \end{center}
  \caption{%
     {\tt Above we have shown the gaps in the point set $\wh{\Gamma}^{\infty}$. The point set is exactly the one shown in \prettyref{fig:example}. We have cut off the image at $240$ (thus the first three bars do not have the same height) and the bin size here is $4\times 10^{-8}$. Hence the bars represent the number of gaps lying in a particular bin. 
       }
   }%
  
  \label{fig:gap}

\end{figure}

\subsection{Proof}

\begin{proof}[Proof of \prettyref{thm:gap distribution}]

Theorem \ref{thm:gap distribution} is a straightforward consequence of Theorem \ref{thm:Sheared} and Theorem \ref{thm:main theorem}. We begin by addressing (\ref{eqn:first stats}), define
  \begin{equation}
    \mathfrak{C}(\mathscr{A}) := \{ (x,y) \in \R \times (0,1] : x\in \mathscr{A} y \} \subset \R^2
  \end{equation}
  and note that
  \begin{equation}
    \frac{p}{q} \in x + \mathscr{A}_t \quad , \quad 0 < q \le Q
  \end{equation}
  is equivalent to
  \begin{equation}
    \Longleftrightarrow (p,q) n_+(\vect{x}) \Phi^t \in \mathfrak{C}(\mathscr{A}).
  \end{equation}

  Therefore for a given $x \in \mathcal{D}$
  \begin{equation}
    P_Q(\mathcal{D},\mathscr{A},k) = \frac{e^{(1-\delta_{\Gamma})t}}{\mu_{N_+}^{PS}(\mathcal{D})} \vol(\{ x \in \mathcal{D} : \left| \mathcal{Z} n_+(x) \Phi^t \cap \mathfrak{C}(\mathscr{A})\right| = k \}).
  \end{equation}
  Applying Theorem \ref{thm:Sheared} then implies
  \begin{equation}
    P(k,\mathscr{A}) = \frac{1}{|m^{BMS}|} m^{BR}(\mathcal{S}_k). \label{eqn:P limit}
    \end{equation}
    where $\mathcal{S}_k = \{ \alpha \in \Gamma \backslash G : |\mathcal{Z} \alpha \cap \mathfrak{C}(\mathscr{A})|=k\}$. 

    Turning now to (\ref{eqn:second stats}). Write
    \begin{align}
      \begin{aligned}
        P_{0} (\mathscr{A},k) &=\lim_{t\to \infty} \frac{\left|\{ r \in \mathcal{F}_Q \cap \mathcal{D}: | \mathcal{Z}n_+(r)\Phi^t \cap \mathfrak{C}(\mathscr{A})| = k \}\right|}{e^{\delta_{\Gamma}t} \mu_{N_+}^{PS}(\mathcal{D})}\\
        &=  \lim_{t\to \infty} \frac{\sum_{r \in \mathcal{F}_Q}\chi_{\mathcal{S}_k}(r,n_+(r)\Phi^t)}{\mu_{N_+}^{PS}(\mathcal{D})e^{\delta_{\Gamma} t}}.
      \end{aligned}
    \end{align}
    Applying \prettyref{thm:main theorem} (after extending it to characteristic functions as done in \cite{Lutsko2018}) gives

    \begin{equation} \label{eqn:P_0 limit}
      P_0(\mathscr{A},k) = \frac{1}{|m^{BMS}|}\int_{\T \times [0,\infty)} \tilde{\chi}_{\mathcal{S}_k}(n_-(w) \Phi^{-r}) e^{\delta_{\Gamma}r}dr d\mu^{PS}(w).
    \end{equation}
    
    Note that the quantity in (\ref{eqn:P limit}) is finite for $k >0$. This was proven in {\cite[Proposition 4.3]{Lutsko2018}}. This does not hold for $k=0$ and is the reason for that restriction in the Theorem. The integral on the right hand side of (\ref{eqn:P_0 limit}) is finite whenever the Burger-Roblin measure is finite. Hence the same {\cite[Proposition 4.3]{Lutsko2018}} implies finiteness of \eqref{eqn:P_0 limit} as well.

\end{proof}

  \section{Explicit Gap Distribution for $\wh{\Gamma}$}
 \label{sec:An Explicit Example}

 We now return to the example, $\wh{\Gamma}$, discussed in Section \ref{sec:Introduction}. First note that Theorem \ref{thm:gap distribution} implies that, in the limit $T \to \infty$, the gap distribution in \eqref{F0 gap} exists for all $s>0$. Our goal is to prove the following Theorem which gives a far more explicit formula for the limiting gap distribution:


\begin{theorem} \label{thm:Gamma_0}
  For $s < s_0 = 7.5$, and $\cI$ a closed interval in $[0,1]$, the limiting gap distribution can be written

  \begin{equation}
    \lim_{T\to \infty} \wh{F}_{T,\cI}(s) =: \wh{F}_{\cI}(s) = F^{1,2}_{\cI}(s) + F^{2,3}_{\cI}(s)
  \end{equation}
  where $F^{1,2}_{\cI}(s)$ and $F^{2,3}_{\cI}(s)$ are explicit integrals over compact regions with respect to a fractal measure (see \eqref{Fij}).

\end{theorem}

 The proof follows the methodology of \cite{RudnickZhang2017}, however there are significant differences. The plan is to break up the gap distribution into a sum over pairs of circles in the initial configuration $\mathcal{K}_0$. Then, using the following lemma (of Rudnick and Zhang) we can express each term in this sum as an integral over a compact area.


\begin{lemma}[{\cite[Lemma 3.5]{RudnickZhang2017}}]\label{lem:circle action}
  Let $M = \smallmat{a}{b}{c}{d} \in \SL(2,\R)$.

  \renewcommand{\labelenumi}{\emph{(\roman{enumi})}}
  \begin{enumerate}
      \item If $c \neq 0$ then under the M\"{o}bius transform $M$, a circle $C(x+yi,y)$ is mapped to
        \begin{equation}
          C\left(\frac{ax+b}{cx+d}+ \frac{y i}{(cx+d)^2},\frac{y}{(cx+d)^2}\right)
        \end{equation}
        if $cx+d \neq 0$, and to the line $\Im \vect{z} = 1/2c^2y$ if $cx+d=0$. When $c=0$, the image circle is
        \begin{equation}
          C\left(\frac{ax+b}{d},\frac{y}{d^2}\right).
        \end{equation}
      \item If $c \neq 0$ then the line $C = \R+yi$ is mapped to
        \begin{equation}
          C\left(\frac{a}{c}+\frac{1}{2c^2y}i, \frac{1}{2c^2y}\right),
        \end{equation}
        and to the line $\R + a^2y i$ if $c =0$.
  \end{enumerate}
  \renewcommand{\labelenumi}{\arabic{enumi}.}

\end{lemma}

 \subsection{Breaking the Gap Distribution Up}
 \label{ss:Breaking the Gap Distribution Up}

  In \cite{RudnickZhang2017} a fundamental observation is that at a given level $T$, the two circles corresponding to neighboring tangencies can be mapped by exactly one or two group elements to a pair in the initial configuration. That is not true here, however the following proposition states that this is the case in the interval $[0,s_0)$.

  \begin{proposition} \label{prop:unique gp ele}
    For any $T$ and $\mathcal{I}$, suppose $\cC$ and $\cC'$ are the circles tangent to $\cC_0$ at $x^j_{T,\mathcal{I}}$ and $x^{j+1}_{T,\mathcal{I}}$. If $T(x^{j+1}_{T,\mathcal{I}}  - x^j_{T,\mathcal{I}}) \le s$ for $s < s_0$ then there exists a $\gamma \in \Gamma$ such that $\cC = \gamma \cC_l$ and $\cC' = \gamma \cC_m$ for $\cC_l \neq \cC_m \in \cK_0$ and neither equal $\cC_0$. Moreover if $\cC$ and $\cC'$ are not tangent then $\gamma$ is unique and if they are tangent then there exist exactly two such $\gamma$.

  \end{proposition}

  \begin{remark}
    The reason we consider $s<s_0$ in Theorem \ref{thm:Gamma_0} is that Proposition \ref{prop:unique gp ele} fails for larger $s$. In words, for larger $s$ some of the gaps considered are not the image of a pair in the initial configuration. To get around this, one could consider a larger initial configuration (i.e consider $\cK$ together with the circles tangent at $1/4$ and $4-1/4$). This would allow Proposition \ref{prop:unique gp ele} to hold for slightly larger $s_0$. Therefore as one considered larger and larger gaps, one would need to consider larger and larger initial configurations and more and more terms in the decomposition below. In this paper we will stick to the case $s_0=7.5$ as it will simplify the following proofs.
  \end{remark}

  For ease of notation, we restrict our attention to circles tangent to $\cC_0$ in $[0,2]$ (i.e beneath $\cC_2$) and adopt the following notation shown in Figure \ref{fig:notation}: first label $\cC_2= \cC^{0}$ and

      \begin{itemize}
        \item The tangencies are labelled by their continued fraction expansions $\alpha_{k_1,\dots,k_i}^{(i)} = [0;4k_1,\dots 4k_i]$.   
        \item The associated circles are labelled $\mathcal{C}^{(i)}_{k_1,\dots,k_i}$. 
        \item The diameter of each circle is similarly labelled $h^{(i)}_{k_1,\dots,k_i}$.
      \end{itemize}
      Thus, each circle $\mathcal{C}^{(i)}_{k_1,\dots,k_i}$ is the \emph{child} of the circle $\mathcal{C}^{(i-1)}_{k_1,\dots,k_{i-1}}$ (to which it is tangent) and the \emph{parent} of $\Z_{\neq 0}$ children -  $\mathcal{C}^{(i+1)}_{k_1,\dots,k_i,k_{i+1}}$ (to which it is also tangent).


      \begin{figure}[ht!]
          \begin{center}    
            \includegraphics[width=1.0\textwidth]{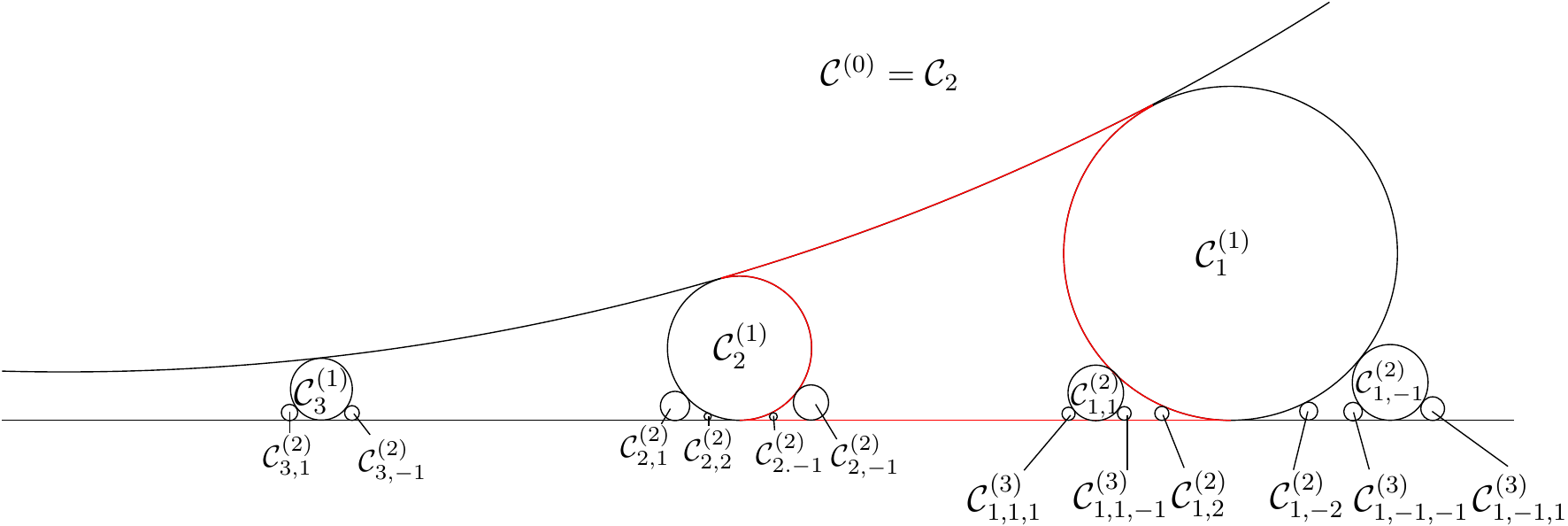}
          \end{center}
          \caption{%
            {\tt The labelling used in this section. For clarity, we only show a portion of the interval and a few circles in $\cK$. The red section is what we call the rectangle $(\cC_1^{(1)}, \cC_2^{(1)}, \cC^{(0)}, \cC_0)$. 
            }
          }%
          
          \label{fig:notation}
       
      \end{figure}

      Define a rectangle to be any collection of circles
      
      \begin{equation}
        \cR = ( \cC^{(i)}_{k_1,\dots,k_{i-1},k_i} , \cC^{(i)}_{k_1,\dots,k_{i-1},k_i\pm 1} , \cC^{(i-1)}_{k_1,\dots,k_{i-1}} , \cC_0 ) \quad , \quad (k_i \neq 0)
      \end{equation}
      where $k_i \pm 1 \neq 0$ (see for example the rectangle in Figure \ref{fig:notation}). A rectangle is thus a pair of neighbors in a generation, the shared parent and the real line. Let $\cR_0$ denote the rectangle $(\cC_0,\cC_1,\cC_2,\cC_3)$ of the initial configuration. The following simple observation is the basis of the proof of Proposition \ref{prop:unique gp ele}.


  \begin{fact} \label{lem:unique rectangle}
    For any rectangle $\mathcal{R}$ there exists a unique $\gamma \in \wh{\Gamma}$
    \begin{equation}
      \cR = \gamma \mathcal{R}_0.
    \end{equation}
  \end{fact}
  The configuration $\cK = \Gamma \cK_0$ where $\cK_0$ is the initial configuration. Since circle inversions send circles to circles preserving tangencies there must be a $\gamma \in \wh{\Gamma}$ sending $\cR_0$ to $\cR$. Moreover the uniqueness follows as we are working in $\PSL(2,\Z)$.

\begin{proof}[Proof of Proposition \ref{prop:unique gp ele}]

  In this proof, given two circles with tangencies $\alpha_1$ and $\alpha_2$ and diameters $h_1$ and $h_2$ we refer to $\abs{\alpha_1 - \alpha_2}$ as the gap associated to them and to $\min\{h_1, h_2\}^{-1}\abs{\alpha_1 - \alpha_2}$ as the \emph{scaled gap} associated to them. Note that if a scaled gap is larger than $s_0$, then the gap \emph{will never} contribute to $\wh{F}_{T,\cI}(s)$ for any $T$. Thus that gap can be ignored. Fact \ref{lem:unique rectangle} implies that Proposition \ref{prop:unique gp ele} follows if we show that all scaled gaps associated to pairs of circles \emph{not in} rectangles are larger than $s_0$.

\begin{steps}[leftmargin=0cm,itemindent=.5cm,labelwidth=\itemindent,labelsep=0cm,align=left]

    \item \label{subgaps} \hspace{2mm} 
      The scaled gap associated to a pair of \emph{non-tangent} circles \emph{in a rectangle} has the form
     \begin{equation}\label{this gap}
       \min\{h^{(i)}_{k_1,\dots , k_i},h^{(i)}_{k_1,\dots , k_i \pm 1}\}^{-1} \left| \alpha^{(i)}_{k_1,\dots, k_i} - \alpha^{(i)}_{k_1,\dots, k_i \pm 1} \right|
     \end{equation}
     (again we assume $k_i \pm 1 \neq 0$). 

   \item \label{im:cont fracs} \hspace{2mm} We now use some theory of continued fractions to show that \eqref{this gap} is bounded above $4$. Therefore the gap arising from non-tangent pairs \emph{in} a rectangle is bounded above $4$. Given a tangency $\alpha^{(i)}_{k_1,\dots,k_i} = [0;a_1,\dots a_i]$, let 

       \begin{equation}
         \frac{b_n}{d_n} := [0;a_1,\dots, a_n]
       \end{equation}
       for $n < i$ where $b_n$ and $d_n$ share no common factors. It is a classical exercise to show (see \cite{Khinchin2003}):

       \begin{align}
         b_n &= a_n b_{n-1} +b_{n-2} , \quad &&b_{-2}=0, \quad &&&b_{-1}=1\\
         d_n &= a_n d_{n-1} +d_{n-2} , \quad &&d_{-2}=1, \quad &&&d_{-1}=0\label{denominator}
       \end{align}
       and

       \begin{equation}
         d_n b_{n-1} - d_{n-1}b_n = (-1)^n.
       \end{equation}
       Hence we can write:

       \begin{align}
         \begin{aligned}
          \min\{h^{(i)}_{k_1,\dots , k_i},h^{(i)}_{k_1,\dots , k_i \pm 1}\}^{-1} &\left| \alpha^{(i)}_{k_1,\dots, k_i} - \alpha^{(i)}_{k_1,\dots, k_i \pm 1} \right| \\
          &= \min\{d_{i}',d_i\}^{2} \left | [1;a_1,\dots a_i] - [1;a_1,\dots a_i\pm 4]  \right|\\
          &= \min\{d_{i}',d_i\}^{2}\left| \frac{a_i b_{i-1} + b_{i-2}}{a_i d_{i-1} + d_{i-2}} - \frac{(a_i \pm 4) b_{i-1} + b_{i-2}}{(a_i\pm 4) d_{i-1} + d_{i-2}}  \right|\\
          & = \min\{d_{i}',d_i\}^{2}  \frac{4}{d_{i}d_i'} \ge 4,
                     \end{aligned}
       \end{align}
       where $b_i$ and $d_i$ are respectively the numerator and denominator of $\alpha^{(i)}_{k_1,\dots,k_i}$ and $b^{\prime}_i$ and $d^{\prime}_i$ are the numerator and denominator of $\alpha^{(i)}_{k_1,\dots, k_i\pm 1}$ (and similarly for all $d_j$ and $b_j$).

     \item \label{im:gaps get smaller} \hspace{2mm} 
       Suppose $\cC^{(i)}_{m_1,\dots , m_i} = \cD_1$ and $\cC_{n_1,\dots,n_j}^{(j)} = \cD_2$ are adjacent at time $T$ and do not both belong to a rectangle. For notation we assume $\alpha^{(i)}_{m_1,\dots m_i} < \alpha^{(j)}_{n_1,\dots n_j}$.
       \begin{itemize}
         \item By construction there is a shared ancestor of $\cD_1$ and $\cD_2$, $\cC_{m_1,\dots, m_k}^{(k)} = \cB_1$ (for $k < \min\{ i,j\}$). That is $m_x = n_x$ for all $1\le x \le k$
         \item At the $k+1$-st generation $\cD_1$ is the decendent of $\cC_{m_1,\dots , m_{k+1}} = \cB_3$ and $\cD_2$ is the decendent of $\cC_{n_1,\dots , n_{k+1}}^{(k+1)} = \cB_2$ (see Figure \ref{fig:gaps get smaller}) and $(\cB_1,\cB_2,\cB_3,\cC_0)$ must form a rectangle (otherwise $\cD_1$ and $\cD_2$ are clearly not adjacent at any times).
         \item Lastly it is evident that $\cD_1$ must be the \emph{right-most} decendent of $\cB_3$ of its generation. Thus $\abs{m_l} =1 $ for all $l > k+1$. Moreover $\cD_2$ must be the \emph{left-most} decendent of $\cB_2$ in its generation.
       \end{itemize}
       Motivated by these three geometric facts we adopt the following notation (see Figure \ref{fig:gaps get smaller}). In each generation $l$, we label the left-most decendent of $\cB_2$ by $\cB_{2(l-k)}$. Moreover we label the right-most decendent of $\cB_{3}$ by $\cB_{2(l-k)+1}$. With that notation, all non-tangent adjacent pairs of circles at a given time are of the form $\cB_{x}$, $\cB_{x+1}$ for some $x$.



      \begin{figure}[ht!]
        \begin{center}    
          \includegraphics[width=0.9\textwidth]{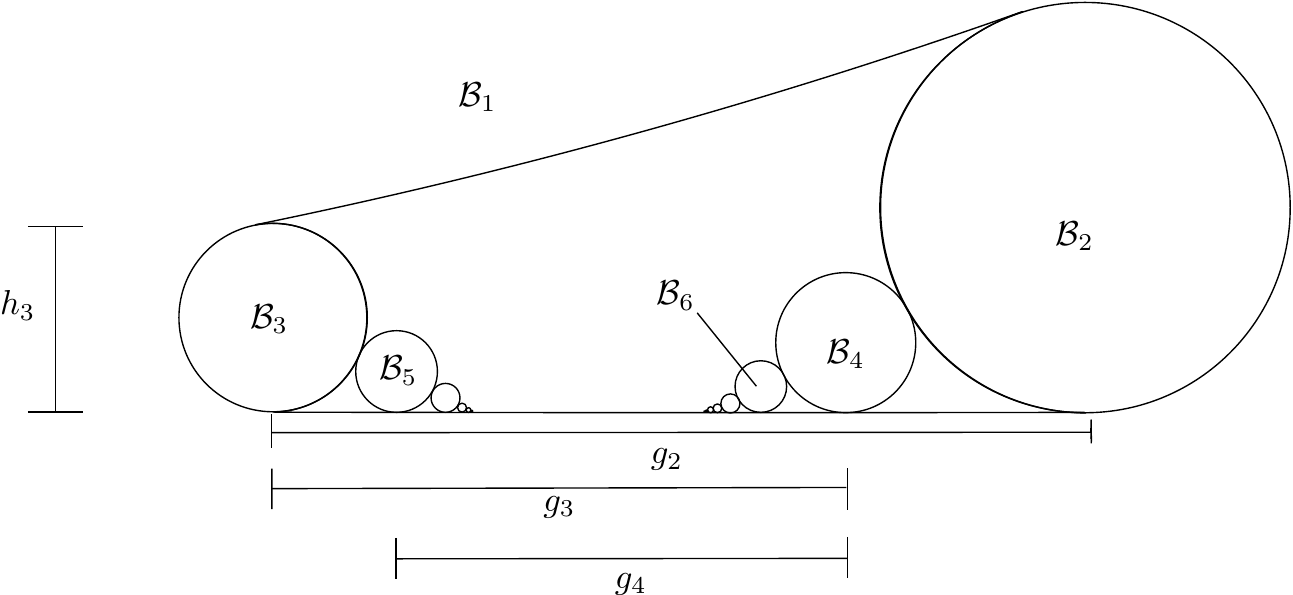}
        \end{center}
        \caption{%
          {\tt Above we show the relevant rectangle, circles and labelling for \ref{im:gaps get smaller}. We are only concerned with the 'innermost circles' in the rectangle. The circles are labelled in decreasing order of size.
          }
        }%
        
        \label{fig:gaps get smaller}

      \end{figure}

      Label the tangency associated to $\cB_i$, $\alpha_i$. Label the diameter of $\cB_i$, $h_i$. We assume (w.l.o.g) $h_1 > h_2 > h_3 = 1$. Label the gap between $\cB_i$ and $\cB_{i+1}$, $g_{i} = |\alpha_{i} - \alpha_{i+1}|$. 

      With this notation, all gaps associated to adjacent (non-tangent) pairs at time $T$ are of the form $g_i$ for $i \ge 2$. We show that $h_{i+1}g_{i}$ (the scaled gap) is larger than $7.5$ for all $i >2$. This will prove the proposition as all gaps associated to non-tangent pairs are of this form.

      Collecting together two facts about $\{\cB_i\}_{i\in \N}$ lead to the bounds. Namely $h_{n+2} \le \frac{h_n}{3^2}$ by \eqref{denominator} and $h_3 \ge \frac{h_2}{4}$. From that, the following sequence of inequalities follow
       \begin{eqnarray}
         h_3 g_2 &\ge& 4 \notag \\
         h_4g_3 &\ge& \left(4-\frac{2}{3}\right)\left(\frac{3}{2}\right)^2 \notag\\
         h_5g_4 &\ge& \left(4- \frac{2}{3}-\frac{1}{3}\right)3^2\\
         h_6g_5 & \ge& \left(4 - \frac{2}{3}-\frac{1}{3}- \frac{2}{9}\right) \left(\frac{9}{2}\right)^2 \notag\\
         h_7g_6 &\ge& \left(4 - \frac{2}{3}-\frac{1}{3}- \frac{2}{9} - \frac{1}{9}\right) 9^2 \notag\\
         \vdots \notag
       \end{eqnarray}
       hence the gap arising from circles which do not form the boundary of a rectangle is at least $7\frac{1}{2}$.

       This proves the proposition with $s_0  =  7\frac{1}{2}$ (this may not be sharp).

\end{steps} 

\end{proof}

Now that we have established this proposition, the argument to prove Theorem \ref{thm:Gamma_0} follows similar lines to Rudnick and Zhang. Note that Proposition \ref{prop:unique gp ele}, implies we can write the gap distribution for $s<s_0$ as

\begin{gather}
  \wh{F}_{T,\cI}(s) = F^{1,2}_{T,\cI}(s) + F^{2,3}_{T,\cI}(s) \label{splitting}\\ 
  F^{i,j}_{T,\cI}(s) := \frac{\# \set{\left.(x_{T,\cI}^l, x_{T,\cI}^{l+1}) \in \Gamma(\alpha_i, \alpha_j) \right| T(x_{T,\cI}^{l+1} - x_{T,\cI}^l) \le s }}{T^{\delta_{\wh{\Gamma}}}}, \label{splitting 2}
\end{gather}
where $\alpha_i$ are the tangencies associated to $\cC_i$ in the initial configuration (the contribution from the tangent pair $(1,3)$ has already been counted from the $(1,2)$ pair because of the overcounting in Proposition \ref{prop:unique gp ele} for gaps associated with tangent pairs).

\subsection{Geometric Description of the Gap Distribution}
\label{ss:Geometric Description}

The Lemma \ref{lem:circle action} and the Proposition \ref{prop:Omega ij} play a crucial role in what follows. As these theorems are taken from \cite{RudnickZhang2017} and are not specific to the subgroup considered, we will not repeat the details here. 

We use Lemma \ref{lem:circle action} to provide conditions under which the image of $\cC_i$ and $\cC_j$ are adjacent at time $T$. Indeed it follows from {\cite[Proposition 4.6]{RudnickZhang2017}} that there exist two regions $\Omega^{1,2}_T$ and $\Omega^{2,3}_T$ such that, for $M = \smallmat{a}{b}{c}{d}$, the image $M(\alpha_i,\alpha_j)$ is an adjacent pair at time $T$ if and only if $(c,d) \in \Omega^{i,j}_T$ (where $(i,j) = (1,2)$ or $(2,3)$).

We define these two regions as subsets of the $cd$-plane $\{(c,d)|c\ge 0\}$:

\renewcommand{\labelenumii}{(\roman{enumii})}
\renewcommand{\labelenumi}{(\alph{enumi})}
\begin{enumerate}
  \item We define $\Omega_T^{1,2}$ to be those $\{(c,d)|c \ge 0\}$ such that
    \begin{gather} \label{Omega 12 1}
      c^2 \le \frac{T}{2} \qquad , \qquad d^2 \le \frac{T}{2}\\
      \label{Omega 12 2}
      (4c+|d|)^2 > \frac{T}{2} 
    \end{gather}

  \item We define $\Omega_T^{2,3}$ to be those $\{(c,d)|c \ge 0\}$ such that
    \begin{gather} \label{Omega 23 1}
      d^2 \le \frac{T}{2} \qquad , \qquad (4c+d)^2 \le \frac{T}{2}. \\
      \label{Omega 23 2}
      \mbox{If } d(4c+d) < 0 \mbox{ then } c^2>\frac{T}{2}.
    \end{gather}
\end{enumerate}
Note that $\Omega_T^{i,j}$ is in both cases a union of convex sets and 

\begin{equation}
  \Omega_T^{i,j} = \sqrt{T}\Omega_1^{i,j}
\end{equation}

Hence we have the following restatement of {\cite[Proposition 4.6]{RudnickZhang2017}} restricted to our context


\begin{proposition}[{\cite[Proposition 4.6]{RudnickZhang2017}}] \label{prop:Omega ij}
  For $\gamma = \smallmat{a_{\gamma}}{b_{\gamma}}{c_{\gamma}}{d_{\gamma}} \in \Gamma$:
  \begin{enumerate}
    \item the circles $\gamma(\cC_1)$ and $\gamma(\cC_2)$ are neighbors in $\cA_T$ if and only if $(c_\gamma,d_\gamma) \in \sqrt{T}\Omega_1^{1,2}$.
    \item the circles $\gamma(\cC_2)$ and $\gamma(\cC_3)$ are neighbors in $\cA_T$ if and only if $(c_\gamma,d_\gamma) \in \sqrt{T}\Omega_1^{2,3}$.
  \end{enumerate}

\end{proposition}

The relative gap condition in \eqref{splitting 2} can now be written (again following {\cite[(18) - (20)]{RudnickZhang2017}}): 

\begin{enumerate}
  \item For $i=1$ and $j =2$

    \begin{equation} \label{gap cond 12}
      c\abs{d} \ge \frac{T}{s}
    \end{equation}
  \item For $i = 2$ and $j =3$

    \begin{equation} \label{gap cond 23}
      \abs{d(4c+d)} \ge \frac{4T}{s}
    \end{equation}
\end{enumerate}

Thus we come to the same conclusion as Rudnick and Zhang that

\begin{equation}\label{F ij Omega}
  F^{i,j}_{T,\cI}(s) = \frac{1}{T^{\delta_{\wh{\Gamma}}}} \# \set{ \gamma = \smallmat{a_{\gamma}}{b_{\gamma}}{c_{\gamma}}{d_{\gamma}} \in \Gamma \;| \; \gamma\alpha_i, \gamma \alpha_j \in \cI, (c_{\gamma},d_{\gamma}) \in \Omega_T^{i,j}(s)}
\end{equation}
for $(i,j) = (1,2), (2,3)$, where $\Omega_T^{i,j}(s)$ is defined to be those elements $(c,d) \in \Omega_T^{i,j}$ satisfying \eqref{gap cond 12} for $(1,2)$ and \eqref{gap cond 23} for $(2,3)$.

Note that $\Omega^{i,j}_T(s)$ are unions of convex, compact sets, and

\begin{equation} \label{Omega s scaling}
  \Omega^{i,j}_T(s) = \sqrt{T}\Omega_1^{i,j}(s)
\end{equation}

\subsection{Limiting Behaviour}
\label{ss:Limiting Behaviour}

To ease notation and remain consistent with \cite{RudnickZhang2017} we reparameterize the geodesic flow

\begin{equation}
  A := \set{ \smallmat{y^{-\frac{1}{2}}}{0}{0}{y^{\frac{1}{2}}} | y>0}
\end{equation}
and set 
\begin{equation}
  A_T : = \set{ \smallmat{y^{-\frac{1}{2}}}{0}{0}{y^{\frac{1}{2}}} | 0<y<T}.
\end{equation}
Note that this is the \emph{backwards geodesic flow} compared with how we defined it in Section \ref{sec:Background Hyp}. Hence we have the coresponding Iwasawa decomposition $\PSL(2,\R) = N_- A K$ (note that $N_-$ is \emph{an expanding horosphere} for this flow). In which case we have the following  Theorem concerning counting points in the orbits of general discrete subgroups, $\Gamma$ (as in the rest of the paper), in bisectors due to Bourgain, Kontorovich and Sarnak


\begin{theorem}[\cite{BourgainKontorovichSarnak2010}]\label{thm:bisector counting} 

  Consider bounded Borel subsets $\Omega_1 \subset N_-$ and $\Omega_2 \subset K$ such that  $\mu^{PS}(\partial (\Omega_1(X_{i})) = \nu_{i}(\partial(\Omega^{-1}(X_{i}^-))) = 0$, then 

  \begin{equation}\label{bisector}
    \lim_{T \to \infty} \frac{ \# (\Gamma \cap \Omega_1 A_T \Omega_2)}{T^{\delta_{\Gamma}}} = \frac{1}{\delta_{\Gamma}\cdot \abs{\BMS}}\mu^{PS}(\Omega_1(X_{i})) \nu_{i}(\Omega_2^{-1}(X_{i}^-)).
  \end{equation}

\end{theorem}

This counting theorem then allows us to prove


\begin{proposition} \label{prop:nice subset}

  Let $\cI$ be an interval, and let $\Omega \subset \set{(c,d) \; | \; c \ge 0}$ be a bounded, convex, compact subset with piecewise smooth boundary. Moreover suppose that in polar coordinates the region $\Omega$ is bounded by two piecewise smooth curves $r_1(\theta) \le r_2(\theta)$ for $\theta \in [\theta_1,\theta_2]$. Then

  \begin{multline}\label{nice subset}
    \# \set{ \left.\gamma = \smallmat{\ast}{\ast}{c_{\gamma}}{d_{\gamma}} \in \Gamma_{\infty} \backslash \Gamma \;\right| \;x(\gamma) \in \cI, \; (c_{\gamma},d_{\gamma}) \in \sqrt{T} \Omega } \\ \sim \frac{T^{\delta_{\Gamma}}}{\delta_{\Gamma} \abs{\BMS}} \mu^{PS}(\mathcal{I}(X_{i})) \int_{\theta_1}^{\theta_2} \left(r_2^{2\delta_{\Gamma}}(\theta) - r_1^{2\delta_{\Gamma}}(\theta)\right) d\nu_{i}(\theta)
  \end{multline}
  as $T \to \infty$, where $d\nu_{i}(\theta) = d\nu_{i}(k(\theta) X_{i})$ and we have written $\gamma$ in $N_-AK$ coordinates as $x(\gamma)a(\gamma)k(\gamma)$.

\end{proposition}

\begin{proof}
  The proof follows the same lines as {\cite[Proposition 5.3]{RudnickZhang2017}}. First we note that using the Iwasawa decomposition of $\gamma$, we have $d_{\gamma}=y^{1/2}\cos\theta$, $c_{\gamma}=y^{1/2}\sin\theta$. Therefore $(y^{1/2}, \theta)$ give a polar coordinate decomposition of the plane.  The rest of the argument follows from a Riemann sum approximation which works equally well when working with $\nu_{i}$. 

  Split the interval $I = [\theta_1,\theta_2]$ into separate equally spaced intervals $\{I_i\}_{i=1}^n$. Take $\theta_{1,i}^+$, and $\theta_{1,i}^-$ to be the points in $I_i$ where $r_1$ is maximized (resp. minimized) and $\theta_{2,i}^+$, and $\theta_{2,i}^-$ to be the points at which $r_2$ is maximized (resp. minimized). Now define

  \begin{align}
    \begin{aligned}
      \Omega_n^- &= \bigcup_{i=1}^n I_i \times [ r_1(\theta_{1,i}^-), r_2(\theta_{2,i}^+)]\\
      \Omega_n^+ &= \bigcup_{i=1}^n I_i \times [ r_1(\theta_{1,i}^+), r_2(\theta_{2,i}^-)].
    \end{aligned}
  \end{align}
  Thus $\Omega_n^- \subseteq \Omega \subseteq \Omega_n^+$ and 

  \begin{multline}
    \lim_{n \to \infty} \sum_{i=1}^n\int_{I_i}  \left(r^{2\delta_{\Gamma}}_2(\theta_{2,i}^+) -  r_1^{2\delta_{\Gamma}}(\theta_{1,i}^-)\right)d\nu_{i}(\theta)\\ = \lim_{n \to \infty} \sum_{i=1}^n\int_{I_i}  \left(r^{2\delta_{\Gamma}}_2(\theta_{2,i}^-) -  r_1^{2\delta_{\Gamma}}(\theta_{1,i}^+)\right)d\nu_{i}(\theta)\\ = \int_{\theta_1}^{\theta_2} \left(r_2^{2\delta_{\Gamma}}(\theta) - r_1^{2\delta_{\Gamma}}(\theta)\right) d\nu_{i}(\theta).
  \end{multline}
  For the truncated regions $\Omega^+_n$ and $\Omega^-_n$ the proposition follows readily with the observation that in \eqref{bisector}, the fact that the conformal density is evaluated at $\Omega_2^{-1}$ simply means that the bounds of integration would be $[-\theta_2,-\theta_1]$. However since our group is symmetric this is equal the integral over $[\theta_1,\theta_2]$. From, since \eqref{nice subset} satisfies finite additivity, the proposition follows.

\end{proof}

Summarizing: provided $s \le  s_0 =  7\frac{1}{2}$ the gap distribution at time $T$ can be written

\begin{equation} 
  \wh{F}_{T,\cI}(s) = F^{1,2}_{T,\cI}(s) + F^{2,3}_{T,\cI}(s).
\end{equation}
Moreover we can take the limit as $T \to \infty$ and \eqref{splitting 2} becomes

\begin{equation}
  \wh{F}_{\cI}(s) = F^{1,2}_{\cI}(s) + F^{2,3}_{\cI}(s)
\end{equation}
where, for $(i,j) =(1,2),(2,3)$

\begin{equation} \label{Fij}
  F_{\cI}^{i,j}(s) = \frac{1}{\delta_{\wh{\Gamma}}|\BMS|} \mu^{PS}(\cI(X_{i})) \int^{\theta_2^{i,j}(s)}_{\theta_1^{i,j}(s)} \left( r_2^{i,j}(\theta,s)^{2\delta_{\wh{\Gamma}}} - r_1^{i,j}(\theta,s)^{2\delta_{\wh{\Gamma}}}\right) d\nu_{i}(\theta), 
\end{equation}
where $\left.r_2^{i,j}(\theta,s)\right|_{\theta \in [\theta_1^{i,j}(s),\theta_2^{i,j}(s)]}$ and $\left.r_1^{i,j}(\theta,s)\right|_{\theta \in [\theta_1^{i,j}(s),\theta_2^{i,j}(s)]}$ are the curves in polar coordinates forming the boundary of $\Omega^{i,j}(s)$. 

For convenience define the constant

\begin{equation}
  \kappa := \frac{1}{\delta_{\wh{\Gamma}}|\BMS|} \mu^{PS}(\cI(X_{i}))
\end{equation}

\subsection{Properties of the Limiting Gap Distribution}
\label{ss:Properties of the Limiting Gap Distribution}

Looking first at $\Omega_1^{1,2}$ defined by \eqref{Omega 12 1}, \eqref{Omega 12 2} and \eqref{gap cond 12}, however since $s < s_0 = 7 \frac{1}{2}$, \eqref{Omega 12 2} can be ignored. Hence we have the region (in $(c,d)$-coordinates):

\begin{equation}
  \Omega_1^{1,2}(s) = [0, \frac{1}{\sqrt{2}}] \times [-\frac{1}{\sqrt{2}}, \frac{1}{\sqrt{2}}] \cap \set{(c,d): c\ge \frac{1}{s|d|}}.
\end{equation}
This region is symmetric under reflection across the $y$ axis and since the conformal density in \eqref{Fij} is invariant under this reflection we can consider 

\begin{equation}\label{Omega tilde}
  \tilde{\Omega}_1^{1,2}(s) = [0, \frac{1}{\sqrt{2}}] \times [0, \frac{1}{\sqrt{2}}] \cap \set{(c,d): c\ge \frac{1}{s|d|}}
\end{equation}
instead, and the only difference will be a factor of $2$.

Regarding $\Omega_1^{2,3}(s)$, from \eqref{Omega 23 1} we know that $\Omega_1^{2,3}$ is a subset of the triangle

\begin{equation}
  -\frac{1}{\sqrt{2}} \le d \le \frac{1}{\sqrt{2}} \quad , \quad 0\le c < \frac{1}{4\sqrt{2}}-d
\end{equation}
Moreover \eqref{Omega 23 2} implies that when $d <0$, if $c> - \frac{d}{4}$ then $c > \frac{1}{\sqrt{2}}$, thus $\Omega_1^{2,3} = T_1 \cup T_2$ where

\begin{gather}
  T_1 := \set{(c,d) : c,d \ge 0 \; , \; c< \frac{1}{4\sqrt{2}}-d }\\
  T_2 := \set{(c,d) : c\ge 0 \;, \; -\frac{1}{\sqrt{2}} \le d \le 0 \; , \; c\le -\frac{d}{4}}.
\end{gather}
Now looking at the condition imposed by \eqref{gap cond 23}, it is straightforawd to see that, for $s <8$, $\Omega_1^{2,3}(s)$ does not intersect $T_2$. Hence, for $s< s_0<8$:

\begin{equation} \label{Omega 23 def}
  \Omega^{2,3}_1(s) = \set{ (c,d) \in T_1 \; : \; c \le \frac{1}{sd}- \frac{d}{4}}.
\end{equation}

So far we have established that

\begin{equation}\label{F Omega}
  \wh{F}(s) = \kappa \mu(\Omega^{2,3}_1(s)) + 2\kappa  \mu(\tilde{\Omega}_1^{1,2}(s))
\end{equation}
where, for a general set $A = \set{(r\cos\theta, r\sin \theta): r \in [r_1^A(\theta),r_2^A(\theta)], \theta \in [\theta^A_1,\theta^A_2]}$,

\begin{equation} \label{mu def}
  \mu(A):= \int_{\theta^A_1}^{\theta_2^A}\left(r_2^A(\theta)^{2\delta_{\wh{\Gamma}}} - r_1^A(\theta)^{2\delta_{\wh{\Gamma}}} \right)d\nu_{i}(\theta).
\end{equation}

Thus $\wh{F}(s)$ is explicitly calculated in terms of the fractal measure $\nu_{i}$. Unfortunately this measure is not itself explicit (in that it is defined as the weak limit of a sequence of measures). However it does lend itself to simulations (which we will not do here) and one can calculate certain analytic properties of $\wh{F}$, we present three below:


\begin{proposition}\label{prop:level repulsion}

  $\wh{F}_{\cI}(s) =0$ for all $s< 2$ for any $\cI$. Moreover, \emph{all} gaps are larger than $2$.

\end{proposition}
This is a form of level repulsion and follows from the definitions of $\tilde{\Omega}_1^{1,2}(s)$ and $\Omega_1^{2,3}(s)$ and \eqref{F Omega}. Indeed $\tilde{\Omega}_1^{1,2}(s)$ is empty for $s<2$ and $\Omega_1^{2,3}(s)$ is empty for $s<4$.

$\nu_{i}$ is a fractal measure supported on the limit set. Hence, looking at \eqref{mu def}, if neither $\theta^A_1$ nor $\theta^A_2$ is in $\cL(\Gamma)$ (the support of $\nu_{i}$). Then the derivative of $\wh{F}$ will be easy to calculate:


\begin{proposition} \label{prop:outside LS}
  Suppose $\mathcal{S} \subset (2,s_0)$ is a connected subset such that for all $s \in \mathcal{S}$,  $\theta^{i,j}_1(s)$ and $\theta^{i,j}_2(s) \not\in \cL(\Gamma)$ for $(i,j) =(1,2)$ or $(2,3)$, then 

  \begin{equation}
    P(s) = \wh{F}^{\prime}(s) = \frac{C_{\mathcal{S}}}{s^{\delta_{\wh{\Gamma}}+1}},
  \end{equation}
  where $ 0 \le C_{\mathcal{S}}<\infty$ depends on the region $\mathcal{S}$ but not on $s\in \mathcal{S}$ and is explicit.
\end{proposition}

\begin{proof}
  Let $s_1 = \inf\set{s \in \mathcal{S}}$, in which case, for $s \in \mathcal{S}$ we separate the integral in \eqref{F Omega} and write

  \small
  \begin{eqnarray}
    \wh{F}(s) &=& \kappa\int_{\theta_1^{2,3}(s)}^{\theta_2^{2,3}(s)} \left(r_2^{2,3}(\theta,s)^{2\delta_{\wh{\Gamma}}} - r_2^{2,3}(\theta,s)^{2\delta_{\wh{\Gamma}}}\right) d\nu_{i} (\theta) + 2\kappa\int_{\theta_1^{1,2}(s)}^{\theta_2^{1,2}(s)} \left(r_2^{1,2}(\theta,s)^{2\delta_{\wh{\Gamma}}} - r_2^{1,2}(\theta,s)^{2\delta_{\wh{\Gamma}}}\right) d\nu_{i} (\theta) \notag \\
        &=& \kappa \int_{\theta_1^{2,3}(s_1)}^{\theta_2^{2,3}(s_1)} \left(r_2^{2,3}(\theta)^{2\delta_{\wh{\Gamma}}} - r_2^{2,3}(\theta,s)^{2\delta_{\wh{\Gamma}}}\right) d\nu_{i} (\theta) + 2\kappa\int_{\theta_1^{1,2}(s_1)}^{\theta_2^{1,2}(s_1)} \left(r_2^{1,2}(\theta)^{2\delta_{\wh{\Gamma}}} - r_2^{1,2}(\theta,s)^{2\delta_{\wh{\Gamma}}}\right) d\nu_{i} (\theta)  \notag \\
       && + R(s,\mathcal{S}) \notag
  \end{eqnarray}
  \normalsize
  where we have noted that (by \eqref{Omega tilde} and \eqref{Omega 23 def}), $r_2$ is independent of $s$ . In fact, since on $\mathcal{S}$, $\theta^{i,j}_1(s)$ and $\theta^{i,j}_2(s)$ are outside $\cL(\Gamma)$, $R(s,\mathcal{S})$ is $0$ (as the measure is supported away from the range of integration). Hence, taking a derivative:

  \begin{equation}
    P(s) = - \kappa \int_{\theta_1^{2,3}(s_1)}^{\theta^{2,3}_2(s_1)} \frac{dr_1^{2,3}(\theta,s)^{2\delta}}{ds} d\nu_{i} (\theta) - 2\kappa \int_{\theta_1^{1,2}(s_1)}^{\theta^{1,2}_2(s_1)} \frac{dr_1^{1,2}(\theta,s)^{2\delta}}{ds} d\nu_{i} (\theta).
  \end{equation}
  Moreover, for $s < s_0$ we have that 

  \begin{equation}
    r_1^{1,2}(\theta,s) = \frac{1}{\sqrt{s}} \sqrt{\frac{1}{\cos \theta \sin\theta}} \quad , \quad r_1^{2,3}(\theta,s) = \frac{1}{\sqrt{s}}\sqrt{\frac{1}{\left(\sin\theta \cos\theta + \frac{\cos^2\theta}{4}\right)}}.
  \end{equation}
  Therefore, for $s \in \mathcal{S}$

  \begin{equation}
    P(s) = \frac{\kappa}{s^{\delta_{\wh{\Gamma}}+1}} \left( \int_{\theta_1^{2,3}(s_1)}^{\theta^{2,3}_2(s_1)} \left( \frac{1}{\left(\sin\theta \cos\theta + \frac{\cos^2\theta}{4}\right)} \right)^{\delta_{\wh{\Gamma}}} d\nu_{i} (\theta) + 2 \int_{\theta_1^{1,2}(s_1)}^{\theta^{1,2}_2(s_1)} \left(\frac{1}{\cos \theta \sin\theta}\right)^{\delta_{\wh{\Gamma}}} d\nu_{i} (\theta) \right).
  \end{equation}

\end{proof}

The final analytic property we calculate for $\wh{F}$ is the following Lipschitz condition:


\begin{proposition}\label{prop:Lipschitz}
  $\wh{F}$ is Lipschitz in a neighborhood of $s$ whenever $s \in  [0,4)$

  \begin{equation} \label{Lipschitz}
    \abs{\wh{F}(s) -\wh{F}(s+x)} \le C_{s}x
  \end{equation}
  for some constant $C_{s}<\infty$.
\end{proposition}

\begin{proof}
  $\wh{F}$ is $0$ on $[0,2)$. Moreover Proposition \ref{prop:outside LS} implies the $\wh{F}$ is differentiable when both $\theta_1^{1,2}$ and $\theta_2^{1,2}$ are outside $\cL(\wh{\Gamma})$. Hence we only need to worry about when $\theta_1^{1,2}(s)$ or $\theta_2^{1,2}(s)$ is a parabolic fixed point (since parabolic points are dense in the limit set).

For any $2\le s < 4$ such that $\theta_1^{1,2}(s)$ or $\theta_2^{1,2}(s)$ is a parabolic fixed point:

  \begin{multline}\label{abs dif}
    \abs{\wh{F}(s) -\wh{F}(s+x)} \le C \left| \int_{\theta^{1,2}_2(s)}^{\theta_2^{1,2}(s+x)} \left( r_2^{1,2}(\theta)^{2\delta_{\wh{\Gamma}}} - r_1^{1,2}(\theta,s)^{2\delta_{\wh{\Gamma}}}\right) d\nu_{i}(\theta)\right. \\
    +\left. \int_{\theta^{1,2}_1(s+x)}^{\theta_1^{1,2}(s)} \left( r_2^{1,2}(\theta)^{2\delta_{\wh{\Gamma}}} - r_1^{1,2}(\theta,s)^{2\delta_{\wh{\Gamma}}}\right) d\nu_{i}(\theta)\right|
  \end{multline}
  Plugging in the formula for $r_2^{1,2}$ and $r_1^{1,2}$ and using Corollary \ref{cor:GMF cor} gives that the first term on the right hand side of \eqref{abs dif} is less than

  \begin{equation}
    \le C_s \abs{ \int_{\theta_2^{1,2}(s)}^{\theta^{1,2}_2(s+x)} \theta^{2\delta_{\wh{\Gamma}}-2} \left(\left(\frac{1/\sqrt{2}}{\sin\theta}\right)^{2\delta_{\wh{\Gamma}}} - \left( \frac{1}{(s+x)\cos\theta \sin \theta}\right)^{\delta_{\wh{\Gamma}}}    \right) d\theta}
  \end{equation}
  in the range with which we are concerned we can bound this integral (by adjusting the constant) by

  \begin{equation}
    \le C_s  \int_{\theta_2^{1,2}(s)}^{\theta^{1,2}_2(s+x)} \theta^{2\delta_{\wh{\Gamma}}-2} d\theta.
  \end{equation}
  Evaluating the integral and performing the same analysis on the other term in \eqref{abs dif} gives

  \begin{equation}
    \abs{\wh{F}(s) -\wh{F}(s+x)}  \le C_s \left(\theta_2^{1,2}(s+x)^{2\delta_{\wh{\Gamma}}-1}- \theta_2^{1,2}(s)^{2\delta_{\wh{\Gamma}}-1}   \right) + C_s \left( \theta_1^{1,2}(s)^{2\delta_{\wh{\Gamma}}-1} - \theta_1^{1,2}(s+x)^{2\delta_{\wh{\Gamma}}-1}\right).
  \end{equation}
  Inserting the definition of $\theta_2^{1,2}$ and $\theta_1^{1,2}$ then gives

  \begin{equation}
    \abs{\wh{F}(s) -\wh{F}(s+x)}  \le C_s \left(\tan^{-1}(s+x)^{2\delta_{\wh{\Gamma}}-1}- \tan^{-1}(s)^{2\delta_{\wh{\Gamma}}-1}   \right) + C_s \left( \cot^{-1}(s)^{2\delta_{\wh{\Gamma}}-1} - \cot^{-1}(s+x)^{2\delta_{\wh{\Gamma}}-1}\right).
  \end{equation}
  From here, Taylor expanding gives

  \begin{equation} \label{Holder}
    \abs{\wh{F}(s) -\wh{F}(s+x)} \le C \abs{\left( \frac{\pi}{4} + \frac{x}{4}\right)^{2\delta_{\wh{\Gamma}}-1} - \left( \frac{\pi}{4}\right)^{2\delta_{\wh{\Gamma}}-1} } + C\abs{\left( \frac{\pi}{4}\right)^{2\delta_{\wh{\Gamma}}-1} - \left( \frac{\pi}{4} - \frac{x}{4}\right)^{2\delta_{\wh{\Gamma}}-1} }.
  \end{equation}
  Here, expanding again gives us that $\wh{F}$ is Lipschitz.

\end{proof}

  \section{Gauss-Like Measure}
 \label{sec:Gauss-Like Measure}
 As in the previous section this section is restricted to the example $\wh{\Gamma}$. The goal for this section is to derive and study the measure 

 \begin{equation} \label{gauss}
   m^0(E) = C_0 \int_E \int_{-2}^2 \frac{d\mu^{PS}(x)}{\abs{xy-1}^{2\delta_{\wh{\Gamma}}}}  d\mu^{PS}(y).
 \end{equation}
 where $E$ is a Borel set in $\cL(\wh{\Gamma}) \cap (-2,2)$ and $C_0$ is a normalizing constant. In particular we show that this measure is invariant and ergodic for the Gauss map. Then, as a corollary of this ergodicity, we are able to show that the Gauss-Kuzmin statistics on $\cQ_4$ converge to an explicit function. It should be noted that the density in \eqref{gauss} is a normalized eigenfunction for the transfer operator associated to the Gauss map. We shall avoid this zeta functions approach here, however it is a promising avenue for later research.

 \subsection{Setup}
 In \cite{Series1985} Series, for the modular group, shows that one can encode the endpoints of geodesics by a 'cutting sequence' which generates the continued fraction expansions of the endpoints. Moreover she identifies a cross-section of the unit tangent bundle such that the return map to this cross-section corresponds to the (classical) Gauss map on the end point. As an application of this, she shows that the Gauss measure is simply a projection of the Haar measure onto these end points. Thus, because the Haar measure is ergodic for the geodesic flow, the Gauss measure is ergodic for the Gauss map. The goal for this subsection is to construct the analogous measure in our context (for $\wh{\Gamma}$). To do this we will project the BMS measure in the same way and show that the resulting measure is ergodic for the Gauss map (for $\wh{\Gamma}$). In the end we will only be working with this measure, however for those interested in the Appendix, we show how to construct the analogous cutting sequences and cross-section in our context (we omit the formal proofs concerning the commuting diagrams as we do not use them and the details are the same as \cite{Series1985}).

 Throughout this section let $(-2,2)^{\ast} = (-2,2) \setminus \{0\}$. Consider the restriction of Gauss map to the limit set, $\cL(\wh{\Gamma}) = \overline{\cQ_4}$ (where $\overline{\cQ_4}$ denotes the closure):

 \begin{align} 
   \begin{aligned}
     \label{Gauss map}
     T :  & \cL(\wh{\Gamma}) \to \cL(\wh{\Gamma})\\
     & [0;a_1,a_2,\dots] \mapsto [0;a_2,\dots]
   \end{aligned}
 \end{align}
 and its inverse

 \begin{equation} \label{Gauss inv}
   T^{-1}([0;a_1,\dots,a_{n-1} ]) = \bigcup_{k \in 4\Z^{\ast}}[0;k,a_1,\dots,a_{n-1}].
 \end{equation}
 The $\sigma$-algebra associated to this Gauss map is now the Borel $\sigma$-algebra on $\R$ intersected with $\cL(\wh{\Gamma})$. The goal is now to take the Bowen-Margulis-Sullivan measure and project it to obtain a measure on $(-2,2)$. We choose the BMS measure as it is invariant and ergodic under the geodesic flow. Thus after projecting we are left with a measure invariant and ergodic under the Gauss map. The following lemma gives the parameterization, this was used in Sullivan's work \cite{Sullivan1979}, however we include the proof for completeness.


 \begin{lemma}\label{end pt}
   For $u \in T^{1}(\half)$ let $z$ denote the Euclidean midpoint of the geodesic containing $u$ and $t : = \beta_{u^-}(z,u)$ (thus $t$ is the arclength from $z$ to $u$). Then

   \begin{equation}
     d\BMS(u) = \frac{1}{\abs{u^+-u^-}^{2\delta_{\Gamma}}} d\mu^{PS}(u^-) d\mu^{PS}(u^+) dt.
   \end{equation}

 \end{lemma}
 
 \begin{remark}
   Note this Lemma is not specific to the subgroup $\wh{\Gamma}$ and holds for any Bowen-Margulis-Sullivan measure associated to a subgroup considered in this paper.
 \end{remark}
 
 \begin{proof}
   First (recalling $s$ from the definition of $\BMS$ \eqref{BMS}) note

   \begin{eqnarray}
     s &:=& \beta_{u^-}(i,u) \notag\\
       & = & \beta_{u^-}(i,z) +\beta_{u^-}(z,u) \notag\\
       & = & \beta_{u^-}(i,z) + t \notag\\
       & = & \beta_{u^-}(i,i + u^-) +  \beta_{u^-}(i + u^-,z) + t 
   \end{eqnarray}
   Now using the definition of the Busemann function, we note that $ \beta_{u^-}(i + u^-,z)$, is the hyperbolic distance (along the vertical geodesic at $u^-$) between the horoball of height $1$ based at $u^-$ and the horoball of height $\abs{u^+ - u^-}$. Thus

   \begin{equation}\label{s reparam}
     s = t + \beta_{u^-}(i,i + u^-) + \ln \abs{u^+-u^-}.
   \end{equation}
   Similarly

   \begin{equation}\label{u+}
     \beta_{u^+}(i,u) = -t + \beta_{u^+}(i,i + u^+) + \ln \abs{u^+-u^-}.
   \end{equation}
   Therefore, writing out the definition of the Burger Roblin measure and inserting \eqref{s reparam} and \eqref{u+}:

   \begin{align}
     \begin{aligned}
       \BMS(u) & :=  e^{\delta_{\Gamma} s} e^{\delta_{\Gamma}\beta_{u^+}(i,u)} d\nu_{i}(u^-) d\nu_{i}(u^+)ds  \\
       & =  \frac{1}{\abs{u^+ - u^-}^{2\delta_{\Gamma}}} (e^{\delta_{\Gamma}\beta_{u^-}(i,i + u^-)}d\nu_{i}(u^-)) (e^{\delta_{\Gamma}\beta_{u^+}(i,i + u^+)})d\nu_{i}(u^+)) dt \\
       & =  \frac{1}{\abs{u^+-u^-}^{2\delta_{\Gamma}}} d\mu^{PS}(u^-) d\mu^{PS}(u^+) dt
     \end{aligned}
   \end{align}
   where in the last line we insert the definition of $\mu^{PS}$.

 \end{proof}

 To derive the Gauss-type measure (similarly to \cite{Series1985} for the classical Gauss measure) we restrict the BMS measure to the $u^-$ coordinate. Integrating over the $u^+$ coordinate in $(-2,2)$ gives

 \begin{equation}
   \int^2_{-2} \frac{d\mu^{PS}(u^+)}{\abs{u^+-u^-}^{2\delta_{\wh{\Gamma}}}}.
 \end{equation}
 Thus, for a set $E \subset (-\infty,-2) \cup (\infty,2)$
 \begin{equation}
   \int_E\int^2_{-2} \frac{d\mu^{PS}(x)}{\abs{x-y}^{2\delta_{\wh{\Gamma}}}} d\mu^{PS}(y)
 \end{equation}
 is a measure. Changing coordinates and using that $d\mu^{PS}(1/y) = y^{-2\delta_{\wh{\Gamma}}}d\mu^{PS}(y)$ (this follows from \eqref{conformal invariance} and a calculation using the Busemann function) gives, for any set $E \subset (-2,2)^{\ast}$

 \begin{equation}
   m^0(E) := C_0\int_E\int^2_{-2} \frac{d\mu^{PS}(x)}{\abs{xy-1}^{2\delta_{\wh{\Gamma}}}} d\mu^{PS}(y),
 \end{equation}
 where $C_0$ is a normalizing constant. In the next section we show that this is indeed $T$-invariant and ergodic.

   \subsection{Invariance and Ergodicity}
   \label{ss: Invariance and Ergodicity}
   
   \begin{theorem} \label{thm:ergodicity}
     On $(-2,2)^{\ast}$, $m^0$ is $T$-invariant and ergodic.
   \end{theorem}
   
   \begin{proof}
     To prove invariance, let $E \subset (-2,2)^{\ast}$ and consider the measure of its preimage

     \begin{equation} \notag
       m^0(T^{-1}(E)) = C_0 \int_{T^{-1}(E)} \int_{-2}^2 \frac{d\mu^{PS}(x)}{\abs{xy -1}^{2\delta_{\wh{\Gamma}}}} d\mu^{PS}(y)
     \end{equation}
     Plugging in the definition of $T^{-1}(E)$ and changing variables ($d\mu^{PS}(1/y) = y^{-2\delta_{\wh{\Gamma}}}d\mu^{PS}(y)$) together with the fact that the Patterson-Sullivan measure is invariant under translation by $4n$ gives

     \begin{align} \notag
       &= C_0  \sum_{n \in \Z^{\ast}} \int_{E+4n} \left( \int_{-2}^2 \frac{d\mu^{PS}(x)}{\abs{y -x}^{2\delta_{\wh{\Gamma}}}}\right) d\mu^{PS}(y)  \\
       &= C_0 \int_{E} \sum_{n \in \Z^{\ast}}  \int_{-2}^2 \left( \frac{d\mu^{PS}(x)}{\abs{y -x-4n}^{2\delta_{\wh{\Gamma}}}}\right) d\mu^{PS}(y). 
     \end{align}
     If we now change the $x$ variable to $x+4n$ this gives
     \begin{equation} \notag
       = C_0 \int_{E}  \int_{(-\infty,-2) \cup (2,\infty)}\frac{d\mu^{PS}(x)}{\abs{y -x}^{2\delta_{\wh{\Gamma}}}} d\mu^{PS}(y). 
     \end{equation}
     Hence applying the change of variables $x \mapsto x^{-1}$ gives
     \begin{equation} \notag
       = C_0 \int_{E}  \int_{-2}^2\frac{d\mu^{PS}(x)}{\abs{xy -1}^{2\delta_{\wh{\Gamma}}}} d\mu^{PS}(y) = m^0(E). 
     \end{equation}
     
     \vspace{2mm}

     This new measure is ergodic for the Gauss map because the BMS is ergodic for the geodesic flow. However to see this directly note first that the density 

     \[ \rho(y) = \int_{-2}^2 \frac{d\mu^{PS}(x)}{\abs{xy-1}^{2\delta_{\wh{\Gamma}}}}\]
     is bounded on $\cL(\wh{\Gamma})$. Given $a_1,\dots,a_n$ and writing $\frac{p_i}{q_i} = [0;a_1,\dots,a_i]$, define the cylinder sets

     \begin{equation}
       \Delta_n := \left\{ \psi_n(t) := \frac{p_n +p_{n-1}t}{q_n+q_{n-1}t} : 0 \le t \le 1 \right\}.
     \end{equation}
     Note that the sets $\Delta_n \cap \cL(\wh{\Gamma})$ generate the Borel $\sigma$-algebra on $\cL(\wh{\Gamma})$. 

     Now, for any $n >0$, for $s<t \in [0,1]$ we have that there exists a $\gamma \in \wh{\Gamma}$ such that

     \begin{align}
       \begin{aligned}
         \mu^{PS}\left(\left.T^{-n}( [\frac{s}{4},\frac{t}{4}))\right| \Delta_n\right) &\asymp \nu_i\left(\left. T^{-n}( [\frac{s}{4},\frac{t}{4}))  \right| \Delta_n \right)\\
         &= \frac{\nu_i(\gamma [\frac{s}{4},\frac{t}{4}))}{\nu_i(\gamma [0,\frac{1}{4}))}\\
         &=  \frac{\nu_i( [\frac{s}{4},\frac{t}{4}))}{\nu_i( [0,\frac{1}{4}))}\\
       \end{aligned}
     \end{align}
     Therefore, as the above mentioned density is bounded above and below, for any $A \subset \cL(\wh{\Gamma}) \cap (-2,2)^{\ast}$ measurable 
     \begin{equation} 
       \frac{1}{C}m^0(A) \le m^0(\left.T^{-n}(A) \right| \Delta_n) \le C m^0(A).
     \end{equation}

     To conclude, assume $A$ is $T$-invariant, then $\frac{1}{C}m^0(A) \le m^0(\left.A\right| \Delta_n)$. If $m^0(A)>0$, then $\frac{1}{C}m^0(\Delta_n) \le m^0(\Delta_n | A)$. Therefore, since the cylinders $\Delta_n$ generate the Borel $\sigma$-algebra of measurable sets, we have that

     \[ \frac{1}{C}m^0(B) \le m^0(B|A)\]
     for all $B$ measurable. Setting $B = A^c$ implies that $m^0(A^c) =0$ and $m^0(A)=1$. Hence $m^0$ is ergodic.

   \end{proof}

\subsection{Gauss-Kuzmin Statistics}

Given a point $x = [0;a_1,a_2,\dots] \in \R$ ($a_i \in \N$), Gauss considered the following problem (further studied by Kuzmin in 1928): let $\wt{P}_{n,k}(x) = \frac{\#(k,n)}{n}$ where $\#(k,n)$ is the number of $a_i = k$ with $i \le n$. Does there exist a limting distribution for $\wt{P}_{n,k}(x)$? Using the ergodicity of the Gauss measure it is fairly simple to show that for Lebesgue-almost every $x$

\begin{equation}
  \lim_{n \to \infty} \wt{P}_{n,k}(x) =  \frac{1}{\ln(2)}\ln \left( 1+\frac{1}{k(k+2)}\right).
\end{equation}
This distribution is now known as Gauss-Kuzmin statistics. For a detailed description of the original problem and history see {\cite[Section 15]{Khinchin2003}}. The problem has an analogue in our setting.

\vspace{2mm}

For $[0;a_1,a_2,...] = x \in \overline{\cQ_4} \cap (-2,2)$ define $\wh{P}_{n,k}(x) = \frac{\#(k,n)}{n}$ where $\#(k,n)$ is the number of $a_i$ equal $k$ for $i \le n$. For simplicity of notation we assume $k>0$. In that case, writing

   \begin{equation}
     \wh{P}_{n,k}(x) = \frac{1}{n}\sum_{s=0}^{n-1} \chi_{(\frac{1}{k+4},\frac{1}{k}]}(T^s x)
   \end{equation}
   and applying the Birkhoff ergodic theorem for $m^0$ imply:


   \begin{theorem}\label{GK}

     For every positive integer $k$ and $\mu^{PS}$-almost every $x = [0;a_1,\dots] \in \overline{\cQ_4}\cap (-2,2)$
       \begin{equation}
         \wh{P}_k(x) = \lim_{n \to \infty} \wh{P}_{n,k}(x) = m^0\left((\frac{1}{k+4},\frac{1}{k}]\right).
       \end{equation}

   \end{theorem}

  \section*{Appendix - Cutting Sequences for $\wh{\Gamma}$}

Working with $\wh{\Gamma}$ the goal of this section is to show that, given a geodesic with right end point in $(-2,2) \cap \cL(\wh{\Gamma})$ (and left end point in $(-\infty,-2)$) there is a correspondence between the way this geodesic cuts the boundaries of fundamental domains and the continued fraction expansion of the end point. This section is exactly analogous to the Bowen-Series coding for geodesics in $\PSL(2,\R)/\PSL(2,\Z)$.

Let $\xi \in (-2,2) \cap \cL(\wh\Gamma)$ and let $\gamma $ be any geodesic whose right endpoint is $\xi $ and which intersects the line $x=-2$. As this geodesic moves from left to right, it will cut each fundamental domain. Each fundamental domain has two funnels and a cusp. Thus the geodesic will separate one of the three from the others. If the geodesic separates a cusp we write a $c$. If it separates a funnel we write an $l$ or an $r$ depending on whether the funnel is to the left or right of the geodesic. See Figure \ref{fig:coding}.


\begin{figure}[ht!]
  \begin{center}    
    \includegraphics[width=0.9\textwidth]{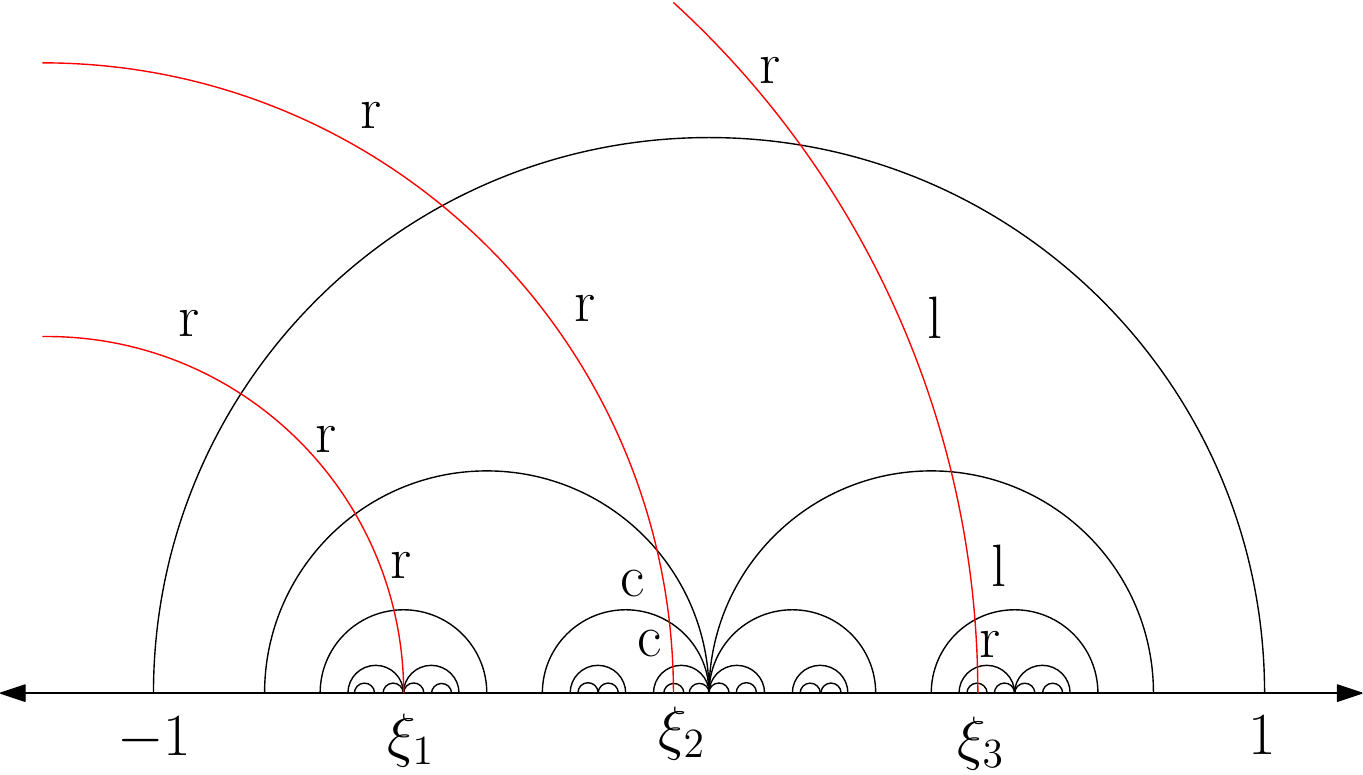}
  \end{center}
  \caption{%
     {\tt In this diagram we show the cutting sequence for 3 different points $\xi_1, \xi_2,\xi_3$. For $\xi_2$, first a funnel is cut off to the \emph{right} of the geodesic, then again a funnel is cut off to the \emph{right}. Then a \emph{cusp} is cut off and then another \emph{cusp}. Thus the first 4 terms in the cutting sequence are $r,r,c,c$.
       }
   }%
  
  \label{fig:coding}

\end{figure}

It is easy to see that the first term in the sequence will always be $r$ and the next term will be $l/r$ after that there will be a sequence of $c$'s followed by the same $l/r$. Thus we end up with a sequence of the form

\begin{equation}
  \xi \mapsto r, q_0, c^{\alpha_0}, q_0, q_1, c^{\alpha_1}, q_1, q_2, c^{\alpha_2},q_2 \dots
\end{equation}
(the sequence is finite if the geodesic ends in a cusp)  where $q_i = l,r$ and $\alpha_i \ge 0$. With that it is fairly easy to see that

\begin{equation}
  \xi = [0; (-1)^{\eta_0}4(\alpha_0+1) , (-1)^{\eta_1}4(\alpha_1+1) , \dots]
\end{equation}
where

\begin{equation}
  \eta_i = \begin{cases}
    0 & \mbox{ if } q_i =l\\
    1 & \mbox{ if } q_i =r
  \end{cases}.
\end{equation}
With that, there is a correspondence between such sequences and geodesics with end points in $(-2,2)$.

In order to identify the appropriate cross-section of $T^1(\Gamma\backslash\half)$ consider the fundamental domain above $i$ and the line connecting $i$ to $\infty$, call it $S$. Given a geodesic $\gamma$ whose left end point is in $(-\infty,-2)\cap \cL(\wh{\Gamma}))$ and whose right endpoint is in $(-2,2)\cap \cL(\wh{\Gamma}))$ consider a point $x \in \gamma \cap S$. We insert $x$ into the cutting sequence of $\gamma$, at its position in the sequence of fundamental domains, resulting in a sequence of the form:

\begin{equation}
  r,q_0 , c^{\alpha_0}, q_0, q_1, c,c,c, x, c, q_1,...
\end{equation}
We say a cutting sequence \emph{changes type} at $x$ if $x$ lies between a $q_i$ and $q_{i+1}$.


\begin{figure}[ht!]
  \begin{center}    
    \includegraphics[width=0.9\textwidth]{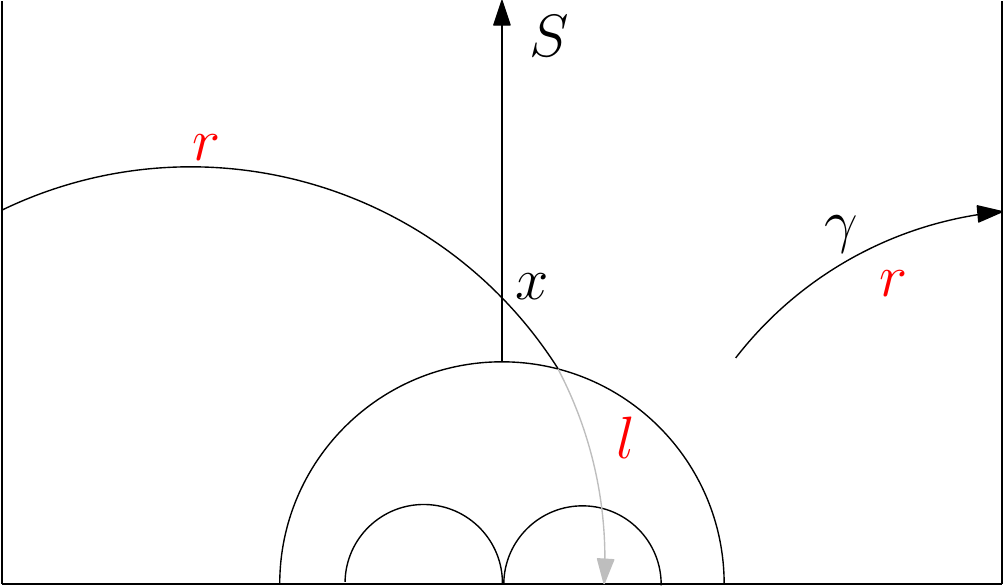}
  \end{center}
  \caption{%
     {\tt In this diagram we show a geodesic in the fundamental domain above $i$, and a point $x \in S \cap \gamma$ such that the cutting sequence for $\gamma$ changes type at $x$. This is because the cutting sequence (pictured in red) with $x$ inserted will read $...,r,r, x , l , ...$.
       }
   }%
  
  \label{fig:cross-section}

\end{figure}

With that, the cross-section $\cC \subset T^1(\Gamma \backslash \half)$ are those points, based at $x \in S$ pointed along geodesics whose cutting sequence changes type at $x$. In that case, the return map to this cross-section corresponds to the Gauss map acting on the end point. For a more formal discussion for the modular group (however the same details apply here) see \cite{Series1985}.

  \section*{Acknowledgements}

  The author was supported by EPSRC Studentship EP/N509619/1 1793795. The author is very grateful to Jens Marklof for his guidance throughout this project. Moreover we thank Florin Boca, Zeev Rudnick, and Xin Zhang for their insightful comments on early preprints.

  \small 
  \bibliographystyle{alpha}
  \bibliography{biblio}

  \vskip2cm

  \hbox{
    \hskip9cm
    \vbox{\hsize=7cm\noindent
      {\sc Authors' address:}
      \\
      School of Mathematics
      \\
      University of Bristol
      \\
      Bristol, BS8 1TW
      \\
      United Kingdom
      \\
      {\tt chris.lutsko@bristol.ac.uk}
      
    }
  }

\end{document}